\documentclass{amsart}

\usepackage{amsmath}
\usepackage{amsthm}
\usepackage{amssymb}
\usepackage{amsfonts}
\usepackage{graphicx}
\usepackage[utf8]{inputenc}
\usepackage{enumerate}
\usepackage{hyperref}
\usepackage{abstract}
\theoremstyle{plain}
\newtheorem{innercustomthm}{Proposition}
\newenvironment{customthm}[1]
  {\renewcommand\theinnercustomthm{#1}\innercustomthm}
  {\endinnercustomthm}
\newtheorem{theorem}{Theorem}[section]
\newtheorem{lemma}[theorem]{Lemma}

\newtheorem{corollary}[theorem]{Corollary}
\newtheorem{proposition}[theorem]{Proposition}

\newtheorem{question}[theorem]{Question}
\newtheorem{problem}[theorem]{Problem}
\theoremstyle{definition}\newtheorem{definition}[theorem]{Definition}

\theoremstyle{definition}
\theoremstyle{definition}\newtheorem{remark}[theorem]{Remark}

\numberwithin{equation}{section}

\newcommand{\R}{\mathbb{R}}

\newcommand{\RR}{\mathbb{R}}
\newcommand{\bd}{\begin{definition}}
\newcommand{\ed}{\end{definition}}

\newcommand{\mc}{\mathcal}
\newcommand{\f}{\bar}

\newcommand{\un}{([0,1]^{<\omega_1}_{\searrow 0},<_{altlex})}
\newcommand{\uni}{[0,1]^{<\omega_1}_{\searrow 0}}
\newcommand{\concat}{%
  \mathbin{\raisebox{1ex}{\scalebox{.7}{$\frown$}}}%
}
\newcommand{\concatt}{%
  \mathbin{\raisebox{1ex}{\scalebox{.9}{$\frown$}}}%
}

\begin{document}

\title[Linearly ordered families of Baire class $1$ functions]{Characterization of order types of pointwise linearly ordered families of Baire class $1$ functions}

\author{M\'arton Elekes}
\address{Alfr\'ed R\'enyi Institute of Mathematics, Hungarian Academy of Sciences,
PO Box 127, 1364 Budapest, Hungary and E\"otv\"os Lor\'and
University, Institute of Mathematics, P\'azm\'any P\'eter s. 1/c,
1117 Budapest, Hungary}
\email{elekes.marton@renyi.mta.hu}
\urladdr{www.renyi.hu/ \!$\tilde{}$ \!\!emarci}

\thanks{Both authors were supported by the National Research, Development and
	Innovation Office – NKFIH, grants no. 113047 and 104178.}

\author{Zolt\'an Vidny\'anszky}
\address{Alfr\'ed R\'enyi Institute of Mathematics, Hungarian Academy of Sciences, PO Box 127, 1364 Budapest, Hungary}
\email{vidnyanszky.zoltan@renyi.mta.hu}

\urladdr{www.renyi.hu/ \!$\tilde{}$ \!\!vidnyanz}

\maketitle

\begin{abstract}

In the 1970s M. Laczkovich posed the following problem: Let $\mathcal{B}_1(X)$ denote the set of Baire class $1$ functions defined on an uncountable Polish space $X$ equipped with the pointwise ordering. 
\begin{center}
\emph{Characterize the order types of the linearly ordered subsets of $\mathcal{B}_1(X)$.}
\end{center}
The main result of the present paper is a complete solution to this problem. 

We prove that a linear order is isomorphic to a linearly ordered family of Baire class $1$ functions iff it is isomorphic to a subset of the following linear order that we call $\un$, where $\uni$ is the set of strictly decreasing transfinite sequences of reals in $[0, 1]$ with last element $0$, and $<_{altlex}$, the so called \emph{alternating lexicographical ordering}, is defined as follows: if $(x_\alpha)_{\alpha\le \xi}, (x'_\alpha)_{\alpha\le \xi'} \in \uni$ are distinct, and $\delta$ is the minimal ordinal where the two sequences differ then we say that 
\[
(x_\alpha)_{\alpha\le \xi} <_{altlex} (x'_\alpha)_{\alpha\le \xi'} \iff
(\delta \text{ is even and } x_{\delta}<x'_{\delta}) \text{ or } (\delta \text{ is odd and } x_{\delta}>x'_{\delta}).
\]

Using this characterization we easily reprove all the known results and answer all the known open questions of the topic.
\end{abstract}

\tableofcontents

\section{Introduction}

Let $\mathcal{F}(X)$ be a class of real valued functions defined on a Polish
space $X$, e.g. $C(X)$, the set of continuous functions. The natural partial ordering on this space is the pointwise ordering
$<_p$, that is, we say that $f<_pg$ if for every $x \in X$ we have $f(x) \leq
g(x)$ and there exists at least one $x$ such that $f(x)<g(x)$. If we would like to understand the structure of this partially ordered set (poset), the first step is to 
describe its linearly ordered subsets. 

For example, if $X = [0,1]$ and $\mathcal{F}(X)=\mathcal{C}([0,1])$ then it is a well known result that the possible order types of the linearly ordered subsets of $\mathcal{C}([0,1])$ are the real order types (that is, the order types of the subsets of the reals). Indeed, a real order type is clearly representable by constant functions, and if $\mathcal{L} \subset \mathcal{C}([0,1])$ is a linearly ordered family of continuous functions then (by continuity) $f \mapsto \int_0^1 f$ is a \emph{strictly} monotone map of $\mathcal{L}$ into the reals. 

The next natural class to look at is the class of Lebesgue measurable functions. However, it is not hard to check that the assumption of measurability is rather meaningless here. Indeed, if $\mathcal{L}$ is a linearly ordered family of \emph{arbitrary} real functions and $\varphi : \RR \to \RR$ is a map that maps the Cantor set onto $\RR$ and is zero outside of the Cantor set then $f \mapsto f \circ \varphi$ is a strictly monotone map of $\mathcal{L}$ into the class of Lebesgue measurable functions. 

Therefore it is more natural to consider the class of Borel measurable functions. However, P. Komj\'ath \cite{komjath} proved that
it is already independent of $ZFC$ (the usual axioms of set theory) whether the class of Borel measurable functions contains a strictly increasing transfinite sequence of length $\omega_2$. 

The next step is therefore to look at subclasses of the Borel measurable functions, namely the Baire hierarchy. A function is of \emph{Baire class
$1$} if it is the pointwise limit of continuous functions. The set of (real valued) Baire class $1$ functions defined on a space $X$ will be denoted by $\mathcal{B}_1(X)$. A function is of \emph{Baire class $2$} if it is the pointwise limit of Baire class $1$ functions. Komj\'ath actually also proved that in his above mentioned result the set of Borel measurable function can be replaced by the set of Baire class $2$ functions. This explains why the Baire class $1$ case seem to be the most interesting one.

Back in the 1970s M. Laczkovich \cite{Mik} posed the following problem:

\begin{problem}
\label{p:laczk}
Characterize the order types of the linearly ordered subsets of $(\mc{B}_1(X),<_p)$.
\end{problem}

We will use the following notation:

\begin{definition}
Let $(P,<_P)$ and $(Q,<_Q)$ be two posets. We say that $P$ is \emph{embeddable into
$Q$}, in symbols $(P,<_P) \hookrightarrow (Q,<_Q)$ if there exists a map $\Phi:P \to Q$ such
that for every $p,q \in P$ if $p<_P q$ then $\Phi(p)<_Q \Phi(q)$. (Note that an embedding may not be 1-to-1 in general. However, an embedding of a \emph{linearly} ordered set \emph{is} 1-to-1.) If $(L, <_L)$ is a linear ordering and $(L, <_L) \hookrightarrow (Q,<_Q)$ then we also say that \emph{$L$ is representable in $Q$}. 
\end{definition}

Whenever the ordering of a poset $(P,<_P)$ is clear from the context we will use
the notation $P=(P,<_P)$. Moreover, when $Q$ is not specified, the term ``representable'' will refer to representability in $\mc{B}_1(X)$.

The earliest result that is relevant to Laczkovich's problem is due to Kuratowski. He
showed that for any Polish space $X$ we have $\omega_1,\omega^*_1 \not
\hookrightarrow \mc{B}_1(X)$, or in other words, there is no $\omega_1$-long strictly
increasing or decreasing sequence of Baire class $1$ functions (see \cite[\textsection 24. III.2.]{kur}).

It seems conceivable at first sight that this is the only obstruction, that is, every linearly ordered
set that does not contain $\omega_1$-long strictly 
increasing or decreasing sequences is representable in $\mc{B}_1(\R)$. First, answering a question of Gerlits and Petruska, this conjecture was consistently refuted by P. Komj\'ath \cite{komjath} who showed that no Suslin line (ccc linearly ordered set that is not separable) is representable in
$\mathcal{B}_1(\R)$. Komj\'ath's short and elegant proof uses the very difficult set-theoretical technique of forcing. Laczkovich \cite{laczk} asked if a forcing-free proof exists.

Elekes and Stepr\=ans \cite{marci2} continued this line of research. On the one hand they proved that consistently
Kuratowski's result is a characterization for order types of cardinality
$<\mathfrak{c}$. On the other hand they strengthened Komj\'ath's result by 
constructing in $ZFC$ a linearly ordered set $L$ not containing Suslin lines or $\omega_1$-long strictly increasing or decreasing sequences such that
$L$ is not representable in $\mathcal{B}_1(X)$. 

Among other results, M. Elekes \cite{marci1} proved that if $X$ and $Y$ are
both uncountable $\sigma$-compact or both non-$\sigma$-compact Polish spaces then for a
linearly ordered set $L$ we have $L\hookrightarrow \mathcal{B}_1(X) \iff
L\hookrightarrow \mathcal{B}_1(Y)$. Then he asked if the same holds if $X$ is an uncountable $\sigma$-compact and $Y$ is a non-$\sigma$-compact Polish space. Moreover, he also asked whether the same linearly
ordered sets can be embedded into the set of \emph{characteristic} functions in
$\mc{B}_1(X)$ as into $\mc{B}_1(X)$. Notice that a characteristic function
$\chi_A$ is of Baire class $1$ if and only if $A$ is simultaneously
$F_\sigma$ and $G_\delta$ (denoted by $A \in \mathbf{\Delta}^0_2(X)$, see the Preliminaries section below). Moreover, $\chi_A <_p \chi_B \iff A \subsetneqq B$, hence
the above question is equivalent to whether $L \hookrightarrow (\mc{B}_1(X), <_p)$ implies
$L \hookrightarrow \mathbf{(\Delta}^0_2(X), \subsetneqq)$. He also asked if \emph{duplications} and completions of representable orders are themselves representable, where the duplication of $L$ is $L \times \{0,1\}$ ordered lexicographically. 

Our main aim in this paper is to solve Problem \ref{p:laczk} and consequently answer the above mentioned questions. 
The solution proceeds by constructing a \emph{universal} linearly ordered set for $\mathcal{B}_1(X)$, that is, a linear order that is representable in $\mc{B}_1(X)$ such that every representable linearly ordered set is embeddable into it. Of course such a linear order only provides a useful characterization if it is sufficiently simple combinatorially to work with. We demonstrate this by providing new, simpler proofs of the known theorems 
(including a forcing-free proof of Komj\'ath's theorem), and also by answering the above mentioned open questions.

The universal linear ordering can be defined as follows.

\begin{definition}
 Let $[0,1]^{<\omega_1}_{\searrow 0}$ be the set of strictly decreasing well-ordered
transfinite sequences in $[0,1]$ with last element zero.
 Let $\bar{x}=(x_\alpha)_{\alpha\leq \xi}, \bar{x}'=(x'_\alpha)_{\alpha\leq
\xi'} \in [0,1]^{<\omega_1}_{\searrow 0}$ be distinct and let $\delta$ be the minimal ordinal such that
$x_\delta \not = x'_\delta$. We say that
\[
(x_\alpha)_{\alpha\le \xi} <_{altlex} (x'_\alpha)_{\alpha\le \xi'} \iff
(\delta \text{ is even and } x_{\delta}<x'_{\delta}) \text{ or } (\delta \text{ is odd and } x_{\delta}>x'_{\delta}).
\]
\end{definition}

Now we can formulate our main result.

\begin{theorem}
(Main Theorem) Let $X$ be an uncountable Polish space. Then the following are
equivalent for a linear ordering $(L,<)$:
\begin{enumerate}[(1)]
  
  \item $(L,<) \hookrightarrow (\mc{B}_1(X),<_p)$,
   \item $(L,<) \hookrightarrow ([0,1]^{<\omega_1}_{\searrow 0},<_{altlex})$.
  \end{enumerate}
  In fact, $(\mc{B}_1(X),<_p)$ and  $\un$ are embeddable into
each other.
\end{theorem}

Using this theorem one can reduce every question concerning the linearly ordered
subsets of $\mathcal{B}_1(X)$ to a purely combinatorial problem. We were able to
answer all of the known such questions and we reproved easily the known theorems
as well. The most important results are:

\begin{itemize}
\item Answering a question of Laczkovich \cite{laczk}, we give a new, forcing free proof of Komj\'ath's theorem. (Theorem \ref{t:suslin})
\item The class of ordered sets representable in $\mc{B}_1(X)$ does not depend on the uncountable Polish space $X$. (Corollary \ref{c:baire1same})
\item There exists an embedding $(\mc{B}_1(X),<_p) \hookrightarrow
(\mathbf{\Delta}^0_2(X),\subsetneqq)$, hence a linear ordering is representable by Baire class $1$ functions iff it is representable by Baire class $1$ \emph{characteristic} functions. (Corollary \ref{c:characteristic})
\item The duplication of a representable linearly ordered set is representable. More generally, countable lexicographical products of representable sets are representable. (Corollary \ref{c:duplication} and Theorem \ref{t:productglue})
\item There exists a linearly ordered set that is representable in
$\mc{B}_1(X)$ but none of its completions are representable. (Theorem \ref{t:complete})
\end{itemize}

The paper is organized as follows. In Section \ref{s:main} we first prove that there exists
an embedding $\mathcal{B}_1(X) \hookrightarrow \uni$, then that $\uni \hookrightarrow \mc{B}_1(X)$. The former result heavily builds on a
theorem of Kechris and Louveau. Unfortunately for us, they only consider the case
of compact Polish spaces, while it is of crucial importance in our proof to use their
theorem for arbitrary Polish spaces. Moreover, their proof seems to contain a slight error. 
Hence it was unavoidable to reprove their result, which is the content of Section \ref{s:last}. 
Section \ref{s:known} contains the new
proofs of the known results, while in Section \ref{s:new} we answer the above open questions.
Finally, in Section \ref{s:open} we formulate some new open
problems.

\section{Preliminaries}

Our terminology will mostly follow \cite{kech} and \cite{tod}.

Let $X$ be a \emph{Polish space}, that is, a separable, completely metrizable
topological space. $\mathcal{B}_1(X)$ denotes the set of the pointwise limits of
continuous functions defined on $X$, this is called the class of \emph{Baire class $1$}
functions. 

$USC(X)$ stands for the set of \emph{upper semicontinuous} functions, that
is, the set of functions $f$ for which for every $r \in \R$ the set
$f^{-1}((-\infty,r))$ is open in $X$.  It is easy to see that the infimum of USC
functions is also USC. 

If $\mathcal{F}(X)$ is a class of real valued functions then we will denote by
$b\mathcal{F}(X)$ and $\mc{F}^+(X)$ the set of bounded and nonnegative functions
in $\mathcal{F}(X)$, respectively.

$\mathcal{K}(X)$ will stand for the set of the nonempty compact subsets of $X$ endowed
with the Hausdorff metric. It is well known (see \cite[Section 4.F]{kech}) that
if $X$ is Polish then so is $\mc{K}
(X)$. Moreover, the compactness of $X$ is equivalent to the compactness of
$\mathcal{K}(X)$.

As usual, we denote the \emph{$\xi$th additive and multiplicative Borel classes} of a Polish space $X$ by $\mathbf{\Sigma}^0_\xi(X)$ and $\mathbf{\Pi}^0_\xi(X)$, respectively. We will also use the notation $\mathbf{\Delta}^0_\xi(X)=\mathbf{\Sigma}^0_\xi(X) \cap\mathbf{\Pi}^0_\xi(X)$. We call a set $A$ \emph{ambiguous}, if $A \in \mathbf{\Delta}^0_2(X)$.  Sometimes the following equivalent definition is also used for the first Baire class: $f \in \mc{B}_1(X) \iff$ the preimage of every open set under $f$ is in $\mathbf{\Sigma}^0_2(X)$ (see \cite[24.10]{kech}). This easily implies that a characteristic function $\chi_A$
is of Baire class $1$ if and only if $A \in \mathbf{\Delta}^0_2(X)$. The above equivalent definition also implies that USC functions are of Baire class $1$.

For a function $f:X \to \R$ \emph{the subgraph of $f$} is the set $sgr(f)=\{(x,r) \in X
\times \R: r \le f(x)\}.$ Notice that a function is USC if and only if its
subgraph is closed.

Let $(P,<_p)$ be a poset. We follow the notation of S. Todor\v{c}evi\' c \cite{tod} letting
\[
\sigma P=\{F:\alpha \to P \mid \alpha \text{ is an ordinal, } F \text{ is strictly increasing}\},
\]
that is, $\sigma P$ is the set of well-ordered sequences in $P$.
We will use the notation  $\sigma^*P$ for the reverse well-ordered sequences, i. e.,
\[
\sigma^* P=\{F:\alpha \to P \mid \alpha \text{ is an ordinal, } F \text{ is strictly decreasing}\}.
\]

Then $\sigma^*[0,1]$ is the
set of strictly decreasing well-ordered transfinite sequences of reals in $[0,1]$. 

For a poset $P$, if $\bar{p} \in \sigma^*P$ and the domain of $\bar{p}$ is $\xi$ then we will write $\bar{p}$ as $(p_\alpha)_{\alpha<\xi}$, where $p_\alpha=\bar{p}(\alpha)$. We will
call the ordinal $\xi$ the length of $\bar{p}$, in symbols $l(\bar{p})$.

Let $H$ and $H'$ be two subsets of the linearly ordered set $(L,<_L)$. We will say that $H \leq_L H'$ or $H<_L H'$ if for every $h \in H$ and $h' \in H'$ we have $h \leq_L h'$ or $h<_Lh'$, respectively. A set $H \subset L$ is called \textit{convex} if for every $a,b \in H$ and $c \in L$ with $a<_Lc<_Lb$ we have $c \in H$. An \textit{interval} is a set of the form $[a,\infty)(=\{c:a \leq_L c \})$, $(a,b]$ or $(a,b)$ etc. for some $a,b \in L$. We say that a linearly ordered set is \textit{densely ordered} if it contains no neighboring points, while it is said to be \textit{separable} if it has a countable subset which intersects every nonempty open interval. Finally, $L$ is \textit{nowhere separable}, if $(a,b) \cap L$ is not separable for every $a,b \in L$.

Now if $\bar{p},\bar{p}' \in \sigma^*P$ and $\bar{p} \not \subset \bar{p}',
\bar{p}' \not \subset \bar{p}$ then there exists a minimal ordinal $\delta$ such
that $p_\delta \not =p'_\delta$. This ordinal is denoted by 
$\delta(\bar{p},\bar{p}')$.

Le $\alpha$ be a successor ordinal, then $\alpha-1$ will stand for its
predecessor. Now, since every ordinal $\alpha$ can be uniquely written in the
form $\alpha=\gamma+n$ where $\gamma$ is limit and $n$ is finite, we let
$(-1)^\alpha=(-1)^n$ and refer to the parity of $n$ as the parity of $\alpha$.

A poset $(T,<_T)$ is called a tree if for every $t \in T$ the ordering $<_T$
restricted to the set $\{s:s<_T t\}$ is a well-ordering. We denote by
$Lev_{\alpha}(T)$ the $\alpha$th level of $T$, that is, the set $\{t \in T:
<_T|_{\{s:s<_T t\}} \text{ has order type } \alpha\}$. If $t \in T$ with $t \in Lev_\alpha(T)$ it can be identified with an enumeration (i. e., a strictly increasing bijection) $e_t:\alpha+1 \to \{s \in T: s \leq_T t\}$. So, for an ordinal $\beta$ we can talk about $t|_\beta$ which is the map $e_t|_\beta$. In particular, if $\alpha$ is a successor (note that $t \in Lev_\alpha(T)$), we will denote by $t|_\alpha$ the predecessor of $t$. 

An \emph{$\alpha$-chain} $C$ is a subset
of a tree such that $<_T|_C$ is a well-ordering in type $\alpha$, whereas an
\emph{antichain} is a set that consists of $\leq_T$-incomparable elements. A set $D \subset T$ is called \emph{dense} if for every $t \in T$ there exists a $p \in D$ such that $t \leq_T p$. A set is called \emph{open} if if for every $p \in D$ we have $\{t \in T: t \geq_T p\} \subset D$.

A tree $(T,<_T)$ of cardinality $\aleph_1$ is called an Aronszajn tree, if for
every $\alpha<\omega_1$ we have $|Lev_{\alpha}(T)| \leq \aleph_0$ and $T$
contains no $\omega_1$-chains. An Aronszajn tree is called a Suslin tree if it
contains no uncountable antichains.

A Suslin line is a linearly ordered set that is ccc (it contains no uncountable pairwise disjoint collection of nonempty open intervals) but not separable. 

We will call a poset $(P,<_P)$ $\R$-special ($\mathbb{Q}$-special) if there
exists an embedding $P \hookrightarrow \R$ ($P \hookrightarrow \mathbb{Q}$).

Every ordinal is identified with the set of its predecessors, in particular, $2 = \{0, 1 \}$.

\section{The main result}
\label{s:main}

\subsection{\texorpdfstring{$\mc{B}_1(X) \hookrightarrow([0,1]^{<\omega_1}_{\searrow 0},<_{altlex})$}{B1(X)->([0,1]omega1->0,<altlex)}}

Recall that
\[
[0,1]^{<\omega_1}_{\searrow 0}=\{\bar{x} \in \sigma^*[0,1]: \min \bar{x} =0\}
\]
and also that
for $\bar{x}=(x_\alpha)_{\alpha\leq \xi}, \bar{x}'=(x'_\alpha)_{\alpha\leq
\xi'} \in [0,1]^{<\omega_1}_{\searrow 0}$ distinct and
$\delta=\delta(\bar{x},\bar{x}')$ we say that
\[
(x_\alpha)_{\alpha\le \xi} <_{altlex} (x'_\alpha)_{\alpha\le \xi'} \iff
(\delta \text{ is even and } x_{\delta}<x'_{\delta}) \text{ or } (\delta \text{ is odd and } x_{\delta}>x'_{\delta}).
\]

\begin{theorem}
\label{t:emb} 
Let $X$ be a Polish space. Then $\mc{B}_1(X) \hookrightarrow
[0,1]^{<\omega_1}_{\searrow 0}$.
\end{theorem}

In order to prove the theorem we have to make some preparation. We will use results
of Kechris and Louveau \cite{KL}. They developed a method to decompose a Baire class 1 function into a sum of a transfinite alternating series,
which is analogous to the well known Hausdorff-Kuratowski analysis of
$\mathbf{\Delta}^0_2$ sets.

First we define the generalized sums.
\begin{definition} (\cite{KL})
Suppose that $(f_{\beta})_{\beta<\alpha}$ is a pointwise decreasing sequence of nonnegative
bounded USC functions for an ordinal $\alpha<\omega_1$. Let us define the
\emph{generalized alternating sum} $\sum^*_{\beta<\alpha} (-1)^\beta f_{\beta}$
by induction on $\alpha$ as follows:
\[\scalebox{1.3}{$\Sigma$}_{\beta < 0}^{*} (-1)^\beta
f_{\beta}=0\]
and 
\[\scalebox{1.3}{$\Sigma$}_{\beta < \alpha}^{*} (-1)^\beta
f_{\beta}=\scalebox{1.3}{$\Sigma$}_{\beta<\alpha-1}^* (-1)^\beta
f_{\beta}+(-1)^{\alpha-1} f_{\alpha-1}\] if $\alpha$ is a successor and 
\[\scalebox{1.3}{$\Sigma$}_{\beta < \alpha}^* (-1)^\beta f_{\beta}=\sup \{
\scalebox{1.3}{$\Sigma$}_{\gamma<\beta}^* (-1)^{\gamma} f_{\gamma}:
\beta<\alpha, \beta \text{ even}\}\] if $\alpha>0$ is a limit.
\end{definition}

Every nonnegative bounded Baire class $1$ function can be canonically decomposed
into such a sum. For this we need the notion of upper regularization.

\begin{definition} (\cite{KL}) Let $f:X \to \R$ be a nonnegative bounded function. The \emph{upper
regularization} of $f$ is defined as \[\hat{f}=\inf \{g:f \leq_p g, g \in
\text{USC}(X)\}. \]
\end{definition}
Note that $\hat{f}$ is USC, since the infimum of USC functions is USC. Also,
clearly $\hat{f}=f$ if $f$ is USC. 

\begin{definition} (\cite{KL}) 
 \label{d:konst}
Let
\[g_0=f, f_0=\widehat{g_0},\] 
if $\alpha$ is a successor then let \[g_{\alpha}=f_{\alpha-1}-g_{\alpha-1},
f_{\alpha}=\widehat{g}_{\alpha},\]
if $\alpha>0$ is a limit then let 
\[
 g_{\alpha}= \inf_{\substack{\beta<\alpha \\ \beta \text{ even}}} g_{\beta}
\text{ and } f_{\alpha}=\widehat{g}_{\alpha}. 
\]

Now if there exists a minimal $\xi_f$ such that $f_{\xi_f} \equiv f_{\xi_f+1}$ then let
$\Phi(f)=(f_\alpha)_{\alpha \leq \xi_f}$. 
\end{definition}

Note that we need some results of Kechris and Louveau for arbitrary Polish
spaces, however in \cite{KL} the authors proved the theorems only in the
compact Polish case, although the proofs still work for the general case as
well.
Unfortunately, in our proof the non-$\sigma$-compact statement plays a
significant role, hence we must check the validity of their results on such
spaces. The results used are summarized in Proposition \ref{p:bb} and the proof
can be found in Section \ref{s:last}. Notice that the original proof seems to contain a
small error, but it can be corrected with the same ideas. 
 \begin{proposition} (\cite{KL})
\label{p:bb} Let $X$ be a Polish space and $f \in b\mathcal{B}^+_1(X)$. Then $\Phi(f)$ is defined,
$\Phi(f)\in \sigma^*bUSC^+$ and we have 
\begin{enumerate}[(1)]
	  \item $f=\sum^*_{\beta<\alpha} (-1)^\beta
f_{\beta}+(-1)^{\alpha}g_\alpha$ for every $\alpha \leq \xi_f$,
	  \item $f_{\xi_f} \equiv0,$
      \item $f=\sum^*_{\alpha<\xi_f} (-1)^\alpha f_{\alpha}$.
         \end{enumerate}

\end{proposition}
\begin{proof}See Section \ref{s:last}. \end{proof}
\begin{proposition}\label{p:embpos}
  Let $X$ be a Polish space and $f_0,f_1 \in b\mc{B}^+_1(X)$. Suppose that $f_0<_pf_1$ and let $\Phi(f_0)=(f^0_\alpha)_{\alpha \leq \xi_{f_0}}$
and $\Phi(f_1)=(f^1_\alpha)_{\alpha \leq \xi_{f_1}}$. Then $\Phi(f_0)\not =
\Phi(f_1)$ and if $\delta=\delta(\Phi(f_0),\Phi(f_1))$ then
$f^0_{\delta}<_pf^1_{\delta}$ if $\delta$ is even and
$f^0_{\delta}>_pf^1_{\delta}$ if $\delta$ is odd.
\end{proposition}
\begin{proof}

First notice that if $f_0 \not =f_1$ then by (3) of Proposition \ref{p:bb} we
have that $\Phi(f_0) \not =\Phi(f_1)$.

Let $(g^0_{\beta})_{\beta\leq \xi_{f_0}}$ and $(g^1_{\beta})_{\beta \leq \xi_{f_1}}$ be
the appropriate sequences (used in Definition \ref{d:konst} with
$\widehat{g^i_\beta}=f^i_\beta$). 

We show by induction on $\beta$ that for every even ordinal $\beta \leq \delta$ we have
$g^0_{\beta} \leq_p g^1_{\beta}$  and for every odd ordinal $\beta \leq \delta$ we have
$g^0_{\beta} \geq_p g^1_{\beta}$. 

For $\beta=0$ by definition $g^0_0=f_0$ and $g^1_0=f_1$, so $g^0_0 \leq_p
g^1_0$.

Suppose that we are done for every $\gamma<\beta$.

\begin{itemize}
 \item for limit $\beta$ we have that \[g^0_{\beta}=
\inf_{\substack{\gamma<\beta \\ \gamma \text{ even}}} g^0_{\gamma}\] so by the
inductive hypothesis obviously $g^0_{\beta} \leq_p g^1_{\beta}$.
 \item if $\beta$ is an odd ordinal, since $\beta-1 < \delta$ we have
$f^0_{\beta-1}=f^1_{\beta-1}$ so 
 \[g^0_{\beta}=f^0_{\beta-1}-g^0_{\beta-1} \geq_p
f^0_{\beta-1}-g^1_{\beta-1}=f^1_{\beta-1}-g^1_{\beta-1}=g^1_{\beta}\] 
 by $\beta-1$ being even and using the inductive hypothesis.
 \item if $\beta$ is an even successor, the calculation is similar, using that
$g^0_{\beta-1}  \geq_p g^1_{\beta-1}$ we obtain
 \[g^0_{\beta}=f^0_{\beta-1}-g^0_{\beta-1} \leq_p
f^0_{\beta-1}-g^1_{\beta-1}=f^1_{\beta-1}-g^1_{\beta-1}=g^1_{\beta}.\] 
 
\end{itemize}

Consequently, the induction shows that $g^0_{\delta} \leq_p g^1_{\delta}$ if
$\delta$ is even and $g^0_{\delta} \geq_p g^1_{\delta}$ if $\delta$ is odd.
Therefore, since $\widehat{g^i_{\delta}}=f^i_{\delta}$ we have that
$f^0_{\delta} \leq_p f^1_{\delta}$ if $\delta$ is even and $f^0_{\delta} \geq_p
f^1_{\delta}$ if $\delta$ is odd. But by the definition of $\delta$ it is clear
that $f^0_{\delta}\not =f^1_{\delta}$, hence $f^0_{\delta} <_p f^1_{\delta}$ if
$\delta$ is even and $f^0_{\delta} >_p f^1_{\delta}$ if $\delta$ is odd. This
finishes the proof of Proposition \ref{p:embpos}.
\end{proof}

Now to finish the proof of Theorem \ref{t:emb} we need the following folklore lemma. 

\begin{lemma}
 \label{l:embed}
 There exists an order preserving embedding $\Psi_0: USC^+(X) \hookrightarrow
[0,1]$ where the image of the function $f \equiv 0$ is $0$. In particular, there
is no uncountable strictly monotone transfinite sequence in $USC^+(X)$.
\end{lemma}
\begin{proof}
Fix a countable basis
$\{B_n:n \in \omega\}$ of $X \times [0,\infty)$. Assign to each $f \in USC^+(X)$
the real \[r_f=1-\sum_{B_n \cap sgr(f) =\emptyset} 2^{-n-1}.\]
 If $f<_pg$ then $sgr(f) \subsetneqq sgr(g)$ so, as the subgraph of an USC function is a closed set, there exists an $n \in \omega$ such
that $B_n$ is an open neighborhood of a point in $sgr(g) \setminus sgr(f)$.
Thus, $\{n:B_n \cap sgr(f) =\emptyset\} \supsetneqq \{n:B_n \cap sgr(g)
=\emptyset\}$.  Consequently, $r_f<r_g$.
\end{proof} 
\begin{proof}[Proof of Theorem \ref{t:emb}] 
Let $\Psi:\sigma^*USC^+(X) \to \sigma^*[0,1]$ be the map that applies the above $\Psi_0$ to every coordinate of the sequences in $\sigma^*USC^+(X)$. Thus, $\Psi$ is order preserving coordinate-wise.

Clearly, $h(x)=\frac{1}{\pi} \arctan (x) + 1$ is an order preserving
homeomorphism from $\R$ to $(0,1)$ and for $f \in \mathcal{B}_1(X)$ let $H(f)=h
\circ f$. Composing the functions in $\mc{B}_1(X)$ with $h$ we still have Baire class $1$ functions and this does not effect the pointwise ordering. Thus, $H$ is an order preserving map from $\mc{B}_1(X)$ into $b\mc{B}^+_1(X)$.

Let $\Theta=\Psi \circ \Phi \circ H$. Notice that as $H:\mathcal{B}_1(X) \to
b\mc{B}^+_1(X)$, $\Phi:b\mc{B}^+_1(X) \to \sigma^*bUSC^+(X)$ and  $\Psi:\sigma^*USC(X)
\to \sigma^*[0,1]$, the map $\Theta$ is well defined.

Now, by Lemma \ref{l:embed} we have that $\Psi_0$ maps the constant zero
function to zero and by $(2)$ of Proposition \ref{p:bb} we have that for every function
$f$ its $\Phi$ image ends with the constant zero function. Thus, the
$\Theta$ image of every function $f$ ends with zero. Therefore, $\Theta$ maps into
$[0,1]^{<\omega_1}_{\searrow 0}$.

If $f_0<_pf_1$ are Baire class $1$ functions then clearly $H(f_0) <_p H(f_1)$
hence by Proposition \ref{p:embpos} we have that if
$\delta=\delta(\Phi(H(f_0)),\Phi(H(f_1)))$, then $\Phi
(H(f_0))(\delta)
<_p\Phi(H(f_1))(\delta)$ if $\delta$ is even and $\Phi
(H(f_0))(\delta)>_p\Phi(H(f_1))(\delta)$ if $\delta$ is odd. Since $\Psi$ is
order preserving coordinate-wise, we obtain that $\Theta$ is an order preserving
embedding of $\mathcal{B}_1(X)$ into $([0,1]^{<\omega_1}_{\searrow 0},<_{altlex})$, which
finishes the proof of the theorem.
\end{proof}

\subsection{\texorpdfstring{$([0,1]^{<\omega_1}_{\searrow 0},<_{altlex}) \hookrightarrow\mc{B}_1(X)$}{([0,1]omega1->0,<altlex)->B1(X)}}
\begin{theorem}
\label{t:repr}
The linearly ordered set $([0,1]^{<\omega_1}_{\searrow 0},<_{altlex})$ can be represented by
$\mathbf{\Delta}^0_2$ subsets of $\mathcal{K}([0,1]^2)$ ordered by inclusion.
\end{theorem}

\begin{proof} First we define a map $\Psi:[0,1]^{<\omega_1}_{\searrow 0} \to
\mathcal{K}([0,1]^2)$, basically assigning to each sequence its closure (as a subset
of the unit interval). However, such a map cannot distinguish between continuous
sequences and sequences omitting a limit point. To remedy this we place a line
segment on each limit point contained in the sequence.

Let $\bar{x} \in [0,1]^{<\omega_1}_{\searrow 0}$, with $\bar{x}=(x_\alpha)_{\alpha \leq
\xi}$. Now let \[\Psi(\bar{x})=\overline{\{(x_\alpha,0):\alpha \leq \xi\}}
\cup\]\[ \bigcup\{\{x_\alpha\} \times [0,x_{\alpha}-x_{\alpha+1}]: \text{ if }
0<\alpha<\xi \text{ and }x_\alpha=\inf\{x_\beta:\beta<\alpha\}\}.\]
\begin{lemma}
$\Psi(\bar{x})$ is a compact set for every $\f{x} \in [0,1]^{<\omega_1}_{\searrow 0}.$
\end{lemma}
\begin{proof}
Clearly, it is enough to show that if $(p_n,q_n) \to (p,q)$ is a convergent
sequence such that
for every $n$ we have
\begin{equation}
\label{e:pq}(p_n,q_n) \in\end{equation}
\[
 \bigcup\{\{x_\alpha\} \times [0,x_{\alpha}-x_{\alpha+1}]: \text{ if }
0<\alpha<\xi \text{ and }x_\alpha=\inf\{x_\beta:\beta<\alpha\}\}\]

then $(p,q) \in \Psi(\bar{x})$. 

Obviously, $p_n=x_{\alpha_n}$ for some ordinals $\alpha_n$. First, if the
sequence $x_{\alpha_n}$ is eventually constant, then there exists an $\alpha$ such
that $p=x_\alpha$ and except for finitely many $n$'s by (\ref{e:pq}) we have
$q_n \in [0,x_\alpha-x_{\alpha+1}]$.
So $(p,q) \in \{x_\alpha\} \times [0,x_\alpha-x_{\alpha+1}] \subset
\Psi(\bar{x})$.

Now if the sequence $(x_{\alpha_n})_{n \in \omega}$ is not eventually constant,
since the  sequence $(x_\alpha)_{\alpha \leq \xi}$ is strictly decreasing and
well-ordered then (passing to a subsequence of $(x_{\alpha_n})_{n \in \omega}$
if necessary) we can suppose that
 $(x_{\alpha_n})_{n \in \omega}$ is a strictly decreasing sequence.

Using the fact that $(x_{\alpha_n})_{n \in \omega}$ is a strictly  decreasing subset of $(x_{\alpha})_{\alpha \leq \xi}$ we obtain that
$x_{\alpha_n}-x_{\alpha_{n}+1} \leq x_{\alpha_n}-p$. Hence from (\ref{e:pq}) we
obtain
\[0 \leq q_n \leq x_{\alpha_n}-x_{\alpha_{n}+1} \leq x_{\alpha_n}-p \to 0\]
so $q_n=0$. Therefore, 
\[(p,q)=(\lim_{n \to \infty }{x_{\alpha_n}},0) \in \overline
{\{(x_\alpha,0):\alpha \leq \xi\}} \subset \Psi(\bar{x}).\]
\end{proof}
Now we define a decreasing sequence of subsets of $\mathcal{K}([0,1]^2)$ for each
$\bar{x}=(x_\alpha)_{\alpha \leq \xi}$ and $\alpha\leq \xi$ as follows:
\begin{equation}
 \mathcal{H}^{\bar{x}}_\alpha=\{\Psi(\bar{z}):\bar{z}|_\alpha=\bar{x}|_\alpha,
z_\alpha \leq x_\alpha\}.
 \label{hdef}
\end{equation}
We will use the following notations for an even ordinal $\alpha \leq \xi$:
\begin{equation}
  \label{kdef}
\mathcal{K}^{\bar{x}}_\alpha=\overline{\mathcal{H}^{\bar{x}}_\alpha}(=\overline{
\{\Psi(\bar{z}):\bar{z}|_\alpha=\bar{x}|_\alpha, z_\alpha \leq x_\alpha\}}),
\end{equation}
 and if $\alpha+1 \leq \xi$ then
 \begin{equation}
  \label{ldef}
\mathcal{L}^{\bar{x}}_{\alpha}=\overline{\mathcal{H}^{\bar{x}}_{\alpha+1}}
(=\overline{\{\Psi(\bar{z}):\bar{z}|_{\alpha+1}=\bar{x}|_{\alpha+1},
z_{\alpha+1} \leq x_{\alpha+1}\}}).
 \end{equation}
Finally, if $\alpha=\xi$ then let $\mathcal{L}^{\bar{x}}_{\alpha}=\emptyset$. So
$\mathcal{K}^{\bar{x}}_\alpha$ and $\mathcal{L}^{\bar{x}}_{\alpha}$ is defined
for every even $\alpha \leq \xi$.

Notice that the sequence $(\overline{\mathcal{H}^{\bar{x}}_\alpha})_{\alpha \leq
\xi}$ is a decreasing sequence of closed sets.

To each $\bar{x}=(x_\alpha)_{\alpha \leq \xi}$ let us assign
\[\mc{A}^{\bar{x}}=\bigcup_{\alpha \leq \xi, \alpha \text{ even}}
(\mathcal{K}^{\bar{x}}_\alpha \setminus \mathcal{L}^{\bar{x}}_\alpha).\]
By \cite[22.27]{kech}, since $\mc{A}^{\bar{x}}$ is a transfinite difference of a
decreasing sequence of closed sets, we have $\mc{A}^{\bar{x}} \in
\mathbf{\Delta}^0_2(\mathcal{K}([0,1]^2))$. 

To overcome some technical difficulties we prove the following lemma.
\begin{lemma}
\label{l:technicallemma}
 Let $\bar{z} \in [0,1]^{<\omega_1}_{\searrow 0}$ and $\beta$ be an ordinal such that $\beta+1
\leq l(\bar{z})$.  
 \begin{enumerate}[(1)]
  \item \label{c:revii}If $K \in \overline{\mathcal{H}^{\bar{z}}_{\beta+1}}$, $\beta$ is a
limit ordinal, $\inf\{z_\gamma:\gamma<\beta\}=z_\beta$ and $l(\bar{z})>\beta+1$
then 
  $(z_\beta,z_{\beta}-z_{\beta+1}) \in K$. 
  \item \label{c:reviii} If $K \in \overline{\mathcal{H}^{\bar{z}}_\beta}$ and $\beta$ is a
successor then $(z_{\beta-1},0) \in K$.
  \item \label{c:reviiii}If $K \in \overline{\mathcal{H}^{\bar{z}}_\beta}$, $\beta$ is a limit
ordinal and $\inf\{z_\gamma:\gamma<\beta\}>z_\beta$ or $\beta$ is a successor then
\begin{enumerate}
	
	\item \label{c:triv1} $K \cap ((z_\beta,\inf\{z_{\gamma}:\gamma<\beta\}) \times [0,1])=\emptyset$
	\item \label{c:triv2} $K \cap (\{\inf\{z_{\gamma}:\gamma<\beta\}\}  \times (0,1])=\emptyset$
\end{enumerate}
(notice that if $\beta$ is a successor then
$\inf\{z_{\gamma}:\gamma<\beta\}=z_{\beta-1}$).

 \end{enumerate}

\end{lemma}
\begin{proof}
	For $(\ref{c:revii})$ and $(\ref{c:reviii})$ just notice that by equation (\ref{hdef}) whenever
$\Psi(\bar{w}) \in \mathcal{H}^{\bar{z}}_\beta$
($\mathcal{H}^{\bar{z}}_{\beta+1}$, respectively) then $\Psi(\bar{w})$ contains
the point $(z_{\beta-1},0)$ (the point $(z_\beta,z_{\beta}-z_{\beta+1})$).
Consequently, every compact set which is in the closure of
$\mathcal{H}^{\bar{z}}_\beta$ (or $\mathcal{H}^{\bar{z}}_{\beta+1}$) contains
the point $(z_{\beta-1},0)$ (the point $(z_\beta,z_{\beta}-z_{\beta+1})$).

In order to see $(\ref{c:reviiii})$ first observe the following: if $U \subset [0,1]^2$ is a relatively open set and $L \in \mathcal{K}([0,1]^2)$ then the set $S=\{K\in \mathcal{K}([0,1]^2): K \cap U=L \cap U\}$ is closed: if $K_n \to K$ with $K_n \in S$ then $K \cap U \supset L \cap U$ is obvious. Now if there was an $x \in (K \cap U) \setminus L$ then there would be a (relatively) open set $V \subset U\setminus L$ with $x \in V$. But then by $x \in K$ we would have that $K_n \cap V \not = \emptyset$ for a large enough $n$, contradicting $K_n \cap V=L \cap V=\emptyset$, so $S$ is indeed closed. 

 Now, by the definition of
$\mathcal{H}^{\bar{z}}_\beta$ for every $\bar{w}$ such that $\Psi(\f{w}) \in
\mathcal{H}^{\bar{z}}_\beta$ we have
\[
\label{e:revi}
\psi(\f{w}) \cap ((z_\beta,1] \times [0,1])=\psi(\f{z}) \cap ((z_\beta,1] \times [0,1]).
\]

Thus, since the set $((z_\beta,1] \times [0,1])$ is relatively open in $[0,1]^2$ we have \[\overline{\mathcal{H}^{\bar{z}}_\beta} \subset \{K \in \mathcal{K}([0,1]^2):K \cap ((z_\beta,1] \times [0,1])=\psi(\f{z}) \cap ((z_\beta,1] \times [0,1]) \},\]
as the latter set is closed and  contains $\mathcal{H}^{\bar{z}}_\beta$.
In particular, for every $K \in \overline{\mathcal{H}^{\bar{z}}_\beta}$ we get that 
 \[K \cap
((z_\beta,\inf\{z_{\gamma}:\gamma<\beta\}) \times [0,1])=\emptyset\]
and 
\[K \cap (\{\inf\{z_{\gamma}:\gamma<\beta\}\}  \times (0,1])= \emptyset \]
hold, which proves the
lemma. \end{proof}

In order to show that $\bar{x} \mapsto \mathcal{A}^{\bar{x}}$ is an embedding it is
enough to prove the following claim.

\textbf{Main Claim.} If $\bar{x}<_{altlex} \bar{y}$ then $\mc{A}^{\bar{x}}
\subsetneqq \mc{A}^{\bar{y}}$.

To verify this we have to distinguish two cases.

{\bf Case 1}. $\delta=\delta(\f{x},\f{y})$ is even.

Then $x_\delta<y_\delta$ and
$\delta+1<l(\bar{y})$. We will show the following lemma.

\begin{lemma}
\label{regielso}
$\mathcal{K}^{\bar{x}}_{\delta} \subsetneqq \mathcal{K}^{\bar{y}}_\delta
\setminus  \mathcal{L}^{\bar{y}}_\delta.$
\end{lemma}
\begin{proof}[Proof of Lemma \ref{regielso}.]
From $x_\delta<y_\delta$ we have
\[\{\Psi(\bar{z}):\bar{z}|_\delta=\bar{x}|_\delta, z_\delta \leq x_\delta\}
\subset \{\Psi(\bar{z}):\bar{z}|_\delta=\bar{x}|_\delta, z_\delta \leq
y_\delta\}\] so $\mathcal{K}^{\bar{x}}_{\delta} \subset
\mathcal{K}^{\bar{y}}_{\delta}$.

First, we prove that 
\begin{equation}
\label{e:kcontained}
\mathcal{K}^{\bar{x}}_{\delta} \subset \mathcal{K}^{\bar{y}}_\delta \setminus 
\mathcal{L}^{\bar{y}}_\delta.
\end{equation}
 Here we have to separate two subcases. 

{\sc Subcase 1.} $\delta$ is a limit ordinal and
$y_\delta=\inf\{y_{\alpha}:\alpha<\delta\}$. 

On the one hand, using $(\ref{c:revii})$ of Lemma \ref{l:technicallemma} (with
$\bar{z}=\bar{y}$ and $\beta=\delta$) we obtain that for every $K \in
\mathcal{L}^{\bar{y}}_\delta(=\overline{\mathcal{H}^{\bar{y}}_{\delta+1}})$ we
have $(y_\delta,y_\delta-y_{\delta+1}) \in K$.

On the other hand, from (\ref{c:triv2}) of Lemma \ref{l:technicallemma} (with
$\bar{z}=\bar{x}$ and $\beta=\delta$) we have that for every $K \in
\mathcal{K}^{\bar{x}}_\delta(=\overline{\mathcal{H}^{\bar{x}}_\delta})$ we have
$K \cap (\{\inf\{x_\alpha:\alpha<\delta\}\} \times (0,1])= \emptyset$. In
particular, as $y_\delta=\inf\{y_\alpha:\alpha<\delta\}=\inf\{x_\alpha:\alpha<\delta\}$, we have
 $(y_\delta,y_\delta-y_{\delta+1}) \not \in K$. So we obtain
$\mathcal{K}^{\bar{x}}_\delta \cap \mathcal{L}^{\bar{y}}_\delta=\emptyset$,
hence by $\mathcal{K}^{\bar{x}}_\delta \subset \mathcal{K}^{\bar{y}}_\delta$ we
have $\mathcal{K}^{\bar{x}}_\delta \subset \mathcal{K}^{\bar{y}}_\delta
\setminus \mathcal{L}^{\bar{y}}_\delta$.

{\sc Subcase 2.} $\delta$ is a limit and
$y_\delta<\inf\{y_{\delta'}:\delta'<\delta\}$ or $\delta$ is a successor.

Using $(\ref{c:reviii})$ of Lemma \ref{l:technicallemma} (with $\bar{z}=\bar{y}$ and
$\beta=\delta+1$) we obtain that every $K \in
\mathcal{L}^{\f{y}}_\delta(=\overline{\mathcal{H}^{\bar{y}}_{\delta+1}})$
contains the point $(y_\delta,0)$. From (\ref{c:triv1}) of Lemma \ref{l:technicallemma}
(with $\bar{z}=\bar{x}$, $\beta=\delta$) we have that for every $K \in 
\mathcal{K}^{\bar{x}}_{\delta}(=\overline{\mathcal{H}^{\bar{x}}_{\delta}})$ the
set $K \cap ((x_\delta,\inf\{x_\alpha:\alpha<\delta\}) \times [0,1])$ is empty. But
$y_\delta \in (x_\delta,\inf\{x_\alpha:\alpha<\delta\})$ so
$\mathcal{K}^{\bar{x}}_\delta \cap \mathcal{L}^{\bar{y}}_\delta=\emptyset$. This
finishes the proof of equation (\ref{e:kcontained}).

Second, in order to prove $\mathcal{K}^{\bar{x}}_{\delta} \not =
\mathcal{K}^{\bar{y}}_\delta \setminus  \mathcal{L}^{\bar{y}}_\delta$ let
$\bar{w}$ be such that $\bar{w}|_\delta=\bar{x}|_\delta$,
$x_\delta,y_{\delta+1}<w_\delta<y_\delta$ and $w_{\delta+1}=0$. Clearly,
$\Psi(\bar{w}) \in \mathcal{K}^{\bar{y}}_\delta$.

By (\ref{c:triv1}) of Lemma \ref{l:technicallemma} (used for $\bar{z}=\bar{x}$ and
$\beta=\delta$) we have that $\Psi(\bar{w}) \in
\mathcal{K}^{\bar{x}}_\delta(=\overline{\mathcal{H}^{\bar{x}}_\delta})$ would
imply $\Psi(\bar{w}) \cap ((x_\delta,\inf\{x_\alpha:\alpha<\delta\}) \times [0,1])=
\emptyset$, but $(w_\delta,0) \in (x_\delta,y_\delta) \times [0,1]$ and
$\inf\{x_\alpha:\alpha<\delta\}=\inf\{y_\alpha:\alpha<\delta\} \geq y_\delta$
which is a contradiction. Hence $\Psi(\bar{w}) \not \in
\mathcal{K}^{\bar{x}}_\delta$. 

Now we prove $\Psi(\bar{w}) \not \in \mathcal{L}^{\bar{y}}_{\delta}$. Suppose
the contrary, then using (\ref{c:triv1}) of Lemma \ref{l:technicallemma} (with
$\bar{z}=\bar{y}$ and $\beta=\delta+1$) one can obtain that for every $K \in
\mathcal{L}^{\bar{y}}_{\delta}(=\overline{\mathcal{H}^{\bar{y}}_{\delta+1}})$
the set $K \cap ((y_{\delta+1},y_\delta) \times [0,1])$ is empty. But clearly
$(w_\delta,0)  \in \Psi(\bar{w}) \cap ((y_{\delta+1},y_{\delta}) \times [0,1])$, a
contradiction. So $\Psi(\bar{w}) \not \in \mathcal{L}^{\bar{y}}_{\delta}$.

Thus, it follows that $\Psi(\bar{w}) \in (\mathcal{K}^{\bar{y}}_\delta \setminus
 \mathcal{L}^{\bar{y}}_\delta) \setminus \mathcal{K}^{\bar{x}}_{\delta}$. From
this and from equation (\ref{e:kcontained}) we can conclude Lemma
\ref{regielso}.
\end{proof}

Now we prove the Main Claim in Case 1. If $\delta'$ is even and
$\delta'<\delta$, the definitions (\ref{kdef}) and (\ref{ldef}) of
$\mathcal{K}^{\bar{y}}_{\delta'}$ and $\mathcal{L}^{\bar{y}}_{\delta'}$ depend
only on $(x_{\alpha})_{\alpha \leq \delta'+1}$ so 
\begin{equation}
\label{e:kequals}
\mathcal{K}^{\bar{x}}_{\delta'}=\mathcal{K}^{\bar{y}}_{\delta'}
\end{equation}
and
\begin{equation}
\label{e:lequals}
\mathcal{L}^{\bar{x}}_{\delta'}=\mathcal{L}^{\bar{y}}_{\delta'}.
\end{equation}
Now, from Lemma \ref{regielso} we have $\mc{A}^{\bar{x}} \subset
\mc{A}^{\bar{y}}$, since for every $K \in \mc{A}^{\bar{x}}$ we have either $K
\in \mathcal{K}^{\bar{x}}_{\delta'} \setminus
\mc{L}^{\f{x}}_{\delta'}=\mathcal{K}^{\bar{y}}_{\delta'}\setminus
\mc{L}^{\f{y}}_{\delta'}$ for some $\delta'<\delta$ or $K \in
\mathcal{K}^{\bar{x}}_{\delta}$. 

Moreover, we claim that using Lemma \ref{regielso} one can prove that
$\mc{A}^{\bar{x}} \subsetneqq \mc{A}^{\bar{y}}$. From the definition of
$\mc{A}^{\f{x}}$, from the fact that the sequence
$(\mathcal{H}^{\bar{x}}_{\alpha})_{\alpha \leq
\xi}=(\mathcal{K}^{\bar{x}}_0,\mathcal{L}^{\bar{x}}_0,\mathcal{K}^{\bar{x}}_1,
\mathcal{L}^{\bar{x}}_1,\dots)$ is decreasing and from equations
(\ref{e:kequals}) and (\ref{e:lequals}) follows that
\[(\mathcal{K}^{\bar{x}}_\delta)^c \cap
\mathcal{A}^{\bar{x}}=\bigcup_{\delta'<\delta, \text{ } \delta' \text{ even}}
\mathcal{K}^{\bar{x}}_{\delta'} \setminus
\mathcal{L}^{\bar{x}}_{\delta'}=\bigcup_{\delta'<\delta, \text{ } \delta' \text{
even}}
\mathcal{K}^{\bar{y}}_{\delta'} \setminus
\mathcal{L}^{\bar{y}}_{\delta'}=(\mathcal{K}^{\bar{y}}_\delta)^c \cap
\mathcal{A}^{\bar{y}}\]
So $\mathcal{A}^{\bar{x}} \subset (\mathcal{K}^{\bar{y}}_\delta)^c  \cup
\mathcal{K}^{\bar{x}}_\delta$. Hence, if $K \in (\mathcal{K}^{\bar{y}}_\delta
\setminus \mathcal{L}^{\bar{y}}_\delta) \setminus \mathcal{K}^{\bar{x}}_\delta$
then 
\[K \in \mathcal{K}^{\bar{y}}_\delta \setminus \mathcal{L}^{\bar{y}}_\delta
\subset \mathcal{A}^{\bar{y}}\]
and 
\[K \not \in  (\mathcal{K}^{\bar{y}}_\delta)^c \cup \mathcal{K}^{\bar{x}}_\delta
\supset \mathcal{A}^{\bar{x}}\]
so indeed, we obtain that the containment is strict, hence we are done with Case
1.

{\bf Case 2.} $\delta=\delta(\f{x},\f{y})$ is odd. 

Then $x_\delta>y_\delta$ and
$\delta+1 < l(\bar{x})$.\\
Notice that as the length of $\bar{x}$ is larger than $\delta+1$, the sets
$\mathcal{K}^{\bar{x}}_{\delta+1}$ and $\mathcal{L}^{\bar{x}}_{\delta+1}$ are
defined.

Now for every even $\delta' \leq \delta-1$ the definition of
$\mathcal{K}^{\bar{x}}_{\delta'}$  and $\mathcal{K}^{\bar{y}}_{\delta'}$ depend
only on $(x_{\alpha})_{\alpha \leq \delta'}=(y_{\alpha})_{\alpha \leq
\delta'}$. Thus for every even $\delta' \leq \delta-1$
\begin{equation}
\label{e:kequals2}
\mathcal{K}^{\bar{x}}_{\delta'}=\mathcal{K}^{\bar{y}}_{\delta'}
\end{equation} and also for every even $\delta' < \delta-1$
\begin{equation}
\label{e:lequals2}
\mathcal{L}^{\bar{x}}_{\delta'}=\mathcal{L}^{\bar{y}}_{\delta'}.
\end{equation}

We will show the following:
\begin{lemma}
\label{l:elsomasodik}
\begin{enumerate}[(1)]
\item \label{c:elsomasodiki} $\mc{K}^{\f{x}}_{\delta-1} \setminus \mathcal{L}^{\bar{x}}_{\delta-1}
\subset \mc{K}^{\f{y}}_{\delta-1} \setminus \mathcal{L}^{\bar{y}}_{\delta-1}$
\item \label{c:elsomasodikii}$ \mathcal{K}^{\f{x}}_{\delta+1} \subset 
\mathcal{K}^{\f{y}}_{\delta-1} \setminus \mathcal{L}^{\f{y}}_{\delta-1}.$
\end{enumerate}
\end{lemma}
\begin{proof}[Proof of Lemma \ref{l:elsomasodik}.]

It is easy to prove (\ref{c:elsomasodiki}): from equation (\ref{e:kequals2}) we get
$\mathcal{K}^{\bar{x}}_{\delta-1}=\mathcal{K}^{\bar{y}}_{\delta-1}$. Moreover,
$\mathcal{L}^{\bar{x}}_{\delta-1} \supset \mathcal{L}^{\bar{y}}_{\delta-1}$,
since 
\[\mathcal{L}^{\bar{x}}_{\delta-1}=\overline{\{\Psi(\bar{z}):\bar{z}|_{\delta}
=\bar{x}|_{\delta}, z_{\delta} \leq x_{\delta}\}} \supset
\overline{\{\Psi(\bar{z}):\bar{z}|_{\delta}=\bar{y}|_{\delta}, z_{\delta}
\leq y_{\delta}\}}=\mathcal{L}^{\bar{y}}_{\delta-1}\]
holds by $x_\delta>y_\delta$. 

Now we show (\ref{c:elsomasodikii}). First,
$\mathcal{K}^{\bar{x}}_{\delta+1} \subset
\mathcal{K}^{\bar{x}}_{\delta-1}=\mathcal{K}^{\f{y}}_{\delta-1}$, using that the
sequence $(\mathcal{K}^{\f{x}}_{\alpha})_{\alpha \leq \delta+1}$ is decreasing. 

So it is suffices to show that $\mathcal{K}^{\bar{x}}_{\delta+1} \cap
\mathcal{L}^{\f{y}}_{\delta-1}= \emptyset$.  Using (\ref{c:triv1}) of Lemma
\ref{l:technicallemma} (with $\bar{z}=\bar{y}$ and $\beta=\delta$) we obtain
that for every $K \in 
\mathcal{L}^{\bar{y}}_{\delta-1}(=\overline{\mathcal{H}^{\bar{y}}_{\delta}})$,
we have $K \cap ((y_{\delta},y_{\delta-1}) \times [0,1])=\emptyset$. 

However, by (\ref{c:reviii}) of Lemma \ref{l:technicallemma} (used with $\bar{z}=\bar{x}$ and
$\beta=\delta+1$) if $K \in \mathcal{K}^{\bar{x}}_{\delta+1}(=\overline{\mathcal{H}^{\bar{x}}_{\delta+1}})$ then
$(x_{\delta},0) \in K$. Therefore, $x_\delta \in (y_\delta,y_{\delta-1})$
implies that the intersection  $\mathcal{K}^{\bar{x}}_{\delta+1} \cap
\mathcal{L}^{\f{y}}_{\delta-1}$ must be empty. So we are done with the lemma.
\end{proof}

Now we prove the Main Claim in Case 2. By definition of $\mc{A}^{\bar{x}}$ and by
the fact that the sequence $(\mathcal{H}^{\bar{x}}_{\alpha})_{\alpha \leq
\xi}=(\mathcal{K}^{\bar{x}}_0,\mathcal{L}^{\bar{x}}_0,\mathcal{K}^{\bar{x}}_1,
\mathcal{L}^{\bar{x}}_1,\dots)$ is decreasing we have that if $K \in
\mc{A}^{\bar{x}}$ then either $K \in \mathcal{K}^{\bar{x}}_{\delta'} \setminus
\mc{L}^{\f{x}}_{\delta'}=\mathcal{K}^{\bar{y}}_{\delta'}\setminus
\mc{L}^{\f{y}}_{\delta'}$ for some even $\delta'<\delta-1$ or $K \in
\mathcal{K}^{\bar{x}}_{\delta-1} \setminus \mathcal{L}^{\bar{x}}_{\delta-1}$ or
$K \in \mathcal{K}^{\f{x}}_{\delta+1}$. Hence using equations (\ref{e:kequals2})
and (\ref{e:lequals2}) and Lemma \ref{l:elsomasodik} we obtain
\begin{equation}
\label{e:axcontained}
\mathcal{A}^{\f{x}} \subset \mathcal{A}^{\f{y}}.
\end{equation}
In order to show that $\mathcal{A}^{\f{x}} \not = \mathcal{A}^{\f{y}}$ it is
enough to find a $\bar{w}$ such that
\begin{equation}
 \Psi(\bar{w}) \in \mathcal{K}^{\bar{y}}_{\delta-1} \setminus
\mathcal{L}^{\bar{y}}_{\delta-1} \subset \mc{A}^{\f{y}} \label{ujelso}
\end{equation} and 
\begin{equation}
 \label{ujmasodik}
 \Psi(\bar{w}) \not \in \mathcal{K}^{\bar{x}}_{\delta+1} \cup
(\mathcal{L}^{\bar{x}}_{\delta-1})^c \supset \mc{A}^{\f{x}}.
\end{equation}

Take  $\bar{w}|_{\delta}=\bar{y}|_{\delta}$ and $w_{\delta}$ such that
$x_{\delta+1},y_{\delta}<w_{\delta}<x_{\delta}$ and $w_{\delta+1}=0$. 

Now, in order to see (\ref{ujelso}) clearly $\Psi(\bar{w}) \in
\mathcal{K}^{\bar{y}}_{\delta-1}$. On the other hand if $K \in
\mathcal{L}^{\bar{y}}_{\delta-1}(=\overline{\mathcal{H}^{\bar{y}}_{\delta}})$ by (\ref{c:triv1}) of Lemma \ref{l:technicallemma} (with
$\bar{z}=\bar{y}$ and $\beta=\delta$)  we have  $K \cap
((y_{\delta},y_{\delta-1}) \times [0,1])= \emptyset$. But
$y_\delta<w_\delta<x_\delta<x_{\delta-1}=y_{\delta-1}$, so $(w_\delta,0) \in
\Psi(\bar{w}) \cap ((y_{\delta},y_{\delta-1}) \times [0,1])$. Therefore,
$\Psi(\f{w})\not \in \mathcal{L}^{\bar{y}}_{\delta-1}$. 

In order to prove (\ref{ujmasodik}) it is obvious that $\Psi(\bar{w}) \in
\mathcal{L}^{\bar{x}}_{\delta-1}$. Now using again (\ref{c:triv1}) of Lemma
\ref{l:technicallemma} (with $\bar{z}=\bar{x}$ and $\beta=\delta+1$) we obtain
that whenever $K \in
\mathcal{K}^{\bar{x}}_{\delta+1}(=\overline{\mathcal{H}^{\bar{x}}_{\delta+1}})$
then $K \cap ((x_{\delta+1},x_\delta) \times [0,1])=\emptyset$. However, $w_{\delta}
\in (x_{\delta+1},x_\delta)$ hence $(w_\delta,0) \in \Psi(\bar{w}) \cap
((x_{\delta+1},x_\delta) \times [0,1])$, so $\Psi(\bar{w})\not \in
\mathcal{K}^{\bar{x}}_{\delta+1}$.

So we can conclude that $\mc{A}^{\bar{x}} \not = \mc{A}^{\f{y}}$. Thus, using
equation (\ref{e:axcontained}) we can finish the proof of the Main Claim in Case
2 and hence we obtain Theorem \ref{t:repr} as well. 
\end{proof}
\subsection{The main theorem}

\begin{theorem}
 \label{t:main}
 (Main Theorem) Let $X$ be an uncountable Polish space. Then the following are
equivalent for a linear ordering $(L,<)$:
 \begin{enumerate}[(1)]
  \item $(L,<) \hookrightarrow (\mc{B}_1(X),<_p)$
  \item $(L,<) \hookrightarrow ([0,1]^{<\omega_1}_{\searrow 0},<_{altlex})$
  \item $(L,<) \hookrightarrow (\mathbf{\Delta}^0_2(X),\subsetneqq)$
  
  \end{enumerate}
  In fact, $\un$, $(\mathbf{\Delta}^0_2(X),\subsetneqq)$ and
$(\mc{B}_1(X),<_p)$ are embeddable into each other.
\end{theorem}
\begin{proof}

 $(\mc{B}_1(X),<_p) \hookrightarrow ([0,1]^{<\omega_1}_{\searrow 0},<_{altlex}):$  Theorem
\ref{t:emb}.
 
 $([0,1]^{<\omega_1}_{\searrow 0},<_{altlex}) \hookrightarrow
(\mathbf{\Delta}^0_2(X),\subsetneqq):$ we proved in Theorem \ref{t:repr} that
$([0,1]^{<\omega_1}_{\searrow 0},<_{altlex}) \hookrightarrow
(\mathbf{\Delta}^0_2(\mathcal{K}([0,1]^2)),\subsetneqq)$. Now, \cite[Theorem
1.2]{marci1} states that the class of linear orderings representable in
$\mathbf{\Delta}^0_2$ coincide for all uncountable $\sigma$-compact Polish
spaces. 
 Hence, if $C$ is the Cantor space, then $([0,1]^{<\omega_1}_{\searrow 0},<_{altlex})
\hookrightarrow (\mathbf{\Delta}^0_2(C),\subsetneqq)$. 
 If $X$ is an uncountable Polish space then there exists a continuous injection
$h:C \to X$. Now, since $h(C)$ is a closed set in $X$ we have that $A \mapsto
h(A)$ is an inclusion-preserving embedding 
$(\mathbf{\Delta}^0_2(C),\subsetneqq) \hookrightarrow
(\mathbf{\Delta}^0_2(X),\subsetneqq)$. Consequently,
$([0,1]^{<\omega_1}_{\searrow 0},<_{altlex}) \hookrightarrow
(\mathbf{\Delta}^0_2(X),\subsetneqq).$
 
 $(\mathbf{\Delta}^0_2(X),\subsetneqq) \hookrightarrow (\mc{B}_1(X),<_p):$ if
$A$ is a $\mathbf{\Delta}^0_2$ set then $\chi_A$ is a Baire class $1$ function
and $A \mapsto \chi_A$ is an order preserving
$(\mathbf{\Delta}^0_2(X),\subsetneqq) \hookrightarrow (\mc{B}_1(X),<_p)$ map.\end{proof}

We immediately obtain the answers to Questions 5.2 and 5.3 from \cite{marci2}. 

\begin{corollary}
\label{c:characteristic}
There exists an embedding $\mathcal{B}_1(X) \hookrightarrow
\mathbf{\Delta}^0_2(X)$, hence a linear ordering is representable by Baire class $1$ functions iff it is representable by Baire class $1$ \emph{characteristic} functions.
 \end{corollary}

The equivalence of (1) and (2), implies that the embeddability of a linearly
ordered set into the set of Baire class $1$ functions does not depend on the
underlying Polish space (provided of course that the Polish space is
uncountable). This result answers Question 1.5 from \cite{marci1} affirmatively.

\begin{corollary}
\label{c:baire1same}
If $X$ and $Y$ are uncountable Polish spaces and $L \hookrightarrow \mathcal{B}_1(X)$ then $L \hookrightarrow \mathcal{B}_1(Y)$.
\end{corollary}

From now on we will simply use the notation $\mathcal{B}_1(X) = \mc{B}_1$.

\section{New proofs of known theorems}
\label{s:known}

In this section we would like to demonstrate the strength and applicability of our characterization by providing new, simpler proofs of the theorems of Kuratowski, Komj\'ath, Elekes and Stepr\=ans. In case of Komj\'ath's result our proof does not use the technique of forcing, which is an answer to a question of Laczkovich \cite{laczk}.

We would like to remark here that the above authors mainly investigated $\mc{B}_1(\R)$ and $\mc{B}_1(\omega^\omega)$, but as we saw in Corollary \ref{c:baire1same} the statements do not depend on the underlying (uncountable) Polish space, so we will state them slightly more generally.  

\subsection{Kuratowski's theorem}

\begin{theorem}
 \label{t:kuratowski} (Kuratowski, \cite[\textsection 24. III.2.]{kur})
 $\omega_1$ and $\omega^*_1$ are not
representable in $\mathcal{B}_1$.
\end{theorem}
\begin{proof}
 By the Main Theorem it is enough to prove that $\omega_1 \not \hookrightarrow
[0,1]^{<\omega_1}_{\searrow 0}$ and $\omega^*_1 \not \hookrightarrow [0,1]^{<\omega_1}_{\searrow 0}$. We
will prove the former statement, the proof of the latter is the same.
 
 Suppose that $(f_{\alpha})_{\alpha<\omega_1}$ is a strictly increasing sequence in
$\uni$. Now we define a sequence
$\{s_\alpha:\alpha<\omega_1\} \subset  \sigma^*[0,1]$ that is strictly increasing with respect to containment. Notice that this will yield a
contradiction, since $\cup_{\alpha<\omega_1} s_\alpha$ would be an
$\omega_1$-long strictly decreasing sequence of reals.  
 
 We define the  sequence $s_\alpha$ by induction on $\alpha$ with the
following properties:
 \begin{equation}
 \label{e:indhyp}
 \text{$l(s_\alpha) = \alpha$ and $\{\gamma:s_\alpha \subset f_\gamma\}$
contains an end segment of $\omega_1$.}
 \end{equation}
  First, $s_0=\emptyset$ clearly works. Now suppose that we are done for every
$\beta<\alpha$.
 
 If $\alpha$ is a limit let $s_\alpha=\cup_{\beta<\alpha} s_\beta$.
 Then \[\{\gamma:s_\alpha \subset f_\gamma\}=\bigcap_{\beta<\alpha}
\{\gamma:s_\beta \subset f_\gamma\}\]
 so the set $\{\gamma:s_\alpha \subset f_\gamma\}$ is the intersection of
countably many sets that contain end segments, hence it contains an end segment. Therefore, \eqref{e:indhyp} holds.
 
 Let $\alpha$ be a successor. Let $S=\{\gamma:s_{\alpha-1} \subset f_\gamma\}$.
If $\gamma,\gamma' \in S$ with $\gamma<\gamma'$ then clearly $f_\gamma
<_{altlex} f_{\gamma'}$. By $s_{\alpha-1} \subset f_\gamma$, $s_{\alpha-1}
\subset f_{\gamma'}$ and $l(s_{\alpha-1})=\alpha-1$ we obtain that
$\delta(f_\gamma,f_{\gamma'}) \geq \alpha-1$. So either
$f_\gamma(\alpha-1)=f_{\gamma'}(\alpha-1)$ or
$f_{\gamma}(\alpha-1)<f_{\gamma'}(\alpha-1) \text{ if } \alpha-1 \text{ is even
and }f_{\gamma}(\alpha-1)>f_{\gamma'}(\alpha-1) \text{ if } \alpha-1 \text{ is
odd.}$
Therefore,
$f_{\gamma}(\alpha-1) \leq f_{\gamma'}(\alpha-1)$ if $\alpha-1$ is
even and $f_{\gamma}(\alpha-1) \geq f_{\gamma'}(\alpha-1) \text{ if } \alpha-1$ is
odd.
 Consequently, the map $\gamma \mapsto f_\gamma(\alpha-1)$ is order preserving from $S$ to the unit interval if $\alpha-1$ is even and order reversing if
$\alpha-1$ is odd. But $S$ contains an end segment by induction, and $[0,1]$ contains no subset of type $\omega_1$ or $\omega^*_1$, hence this map attains a constant value, say $r$ on an end segment. Thus, $s_\alpha =s_{\alpha-1} \concat r$ satisfies (\ref{e:indhyp}).\end{proof}

\subsection{Komj\'ath's theorem}
Komj\'ath \cite{komjath} has shown using forcing that a Suslin line is not
representable in $\mathcal{B}_1(\R)$. Laczkovich \cite{laczk} asked if a forcing-free proof exists. Now we provide such a proof.

\begin{theorem}
\label{t:suslin} (Komj\'ath, \cite{komjath})
A Suslin line is not representable in $\mc{B}_1$.
\end{theorem}

{\sc Notation.} Let $(T,<_T)$ be a tree. We denote by $T|_{succ}$ the set $\{t\in T: t \in Lev_\alpha(T), \alpha \text{ is a successor}\}$ ordered by the restriction of $<_T$. Notice that $T|_{succ}$ is also a tree, but it is not a subtree of $T$. If $t \in \sigma^*[0,1]$ we will use the notation $I_t$ for the set $\{\bar{x} \in \uni:t \subset \bar{x}\}$.

\begin{lemma}
\label{l:suslin}
 Suppose that  $\mathcal{S} \subset [0,1]^{<\omega_1}_{\searrow 0}$ is a nowhere separable Suslin line. Then $\sigma^*[0,1]|_{succ}$
contains a Suslin tree.
\end{lemma}

\begin{proof}
Let \begin{equation}
     \label{e:suslin}
     T=\{t\in \sigma^*[0,1]:|\mc{S} \cap I_t|\geq 2\}.
    \end{equation}
We claim that $(T,\subsetneqq)$ is a Suslin tree. 

First, $T$ is clearly a subtree of $(\sigma^*[0,1],\subsetneqq)$ and  $\sigma^*[0,1]$
does not contain uncountable chains hence this is true for $T$ as well.

Second, let $A \subset T$ be an antichain. Notice that for every pair of
incomparable nodes $t,t'\in T$ the sets $I_{t}$ and $I_{t'}$ are disjoint
subsets of $([0,1]^{<\omega_1}_{\searrow 0},<_{altlex})$, hence $I_{t} \cap \mc{S}$ and
$I_{t'} \cap \mc{S}$ are also disjoint convex sets in $\mc{S}$. By
(\ref{e:suslin}) these convex sets are non-degenerate. Since $A \subset T$ is an
antichain the set $\{I_t \cap \mc{S}:t \in A\}$ is a collection of pairwise
disjoint nonempty convex sets in $\mc{S}$. Using that $\mc{S}$ is nowhere
separable for every $t$ we can select a $J_t \subset I_t$ such that $\mc{S} \cap
J_t$ is a nonempty open convex set. By definition $\mc{S}$ is ccc so the set
$\{J_t \cap \mc{S}:t \in A\}$ is countable. Hence $A$ is countable, showing that
$T$ does not contain uncountable antichains.

Third, it is left to show that $T$ is uncountable. Suppose the contrary.  Notice
first that for every $t \in T$ the set $\{r \in [0,1]:\mc{S}\cap I_{t \concat r}
\not =\emptyset\}$ is countable, otherwise, choosing points $\f{p}_r \in
\mc{S} \cap I_{t\concat r}$ the map $r \mapsto \f{p}_r$ would give an
uncountable real subtype of $\mc{S}$, which is impossible (see \cite[Proposition
3.5]{tod}). Hence, as $T$ is also countable, we can select a countable subset
$D$ of $\mathcal{S}$ with the following property: for every $t \in T$ and $r \in
[0,1]$ such that $\mc{S}\cap I_{t \concat r} \not =\emptyset$ there exists a point
$\f{p} \in D$ such that $\bar{p} \in I_{t \concat r}$.

We claim that $D$ is dense in $\mc{S}$ which will contradict the
non-separability of $\mc{S}$. In order to see this let $J \subset \mc{S}$ be a nonempty
open interval.  By passing to a subinterval of $J$ (using that $\mc{S}$ is
nowhere separable) we can assume that $J$ is of the form $[\bar{x},\bar{y}] \cap
\mc{S}$ with $\bar{x} \not = \bar{y}$. Let $\bar{z} \in (\bar{x},\bar{y}) \cap
\mc{S}$ (such a $\bar{z}$ exists by the fact that $\mc{S}$ is nowhere
separable). Clearly $\bar{x}<_{altlex}\f{z}<_{altlex}\f{y}$. Let
$\delta_{\bar{x}}=\delta(\f{x},\f{z})$ and
$\delta_{\bar{y}}=\delta(\f{y},\f{z})$. Then $l(\f{z}) \geq
\max\{\delta_{\bar{x}},\delta_{\bar{y}}\}+1$ and
\begin{equation}
\label{e:xz}
\f{x}(\delta_{\bar{x}})<\f{z}(\delta_{\bar{x}}) \iff \delta_{\bar{x}} \text{
even and }
\f{z}(\delta_{\bar{y}})<\f{y}(\delta_{\bar{y}}) \iff \delta_{\bar{y}} \text{
even.}
\end{equation}

Suppose that $\delta_{\f{x}} \geq \delta_{\f{y}}$, the proof of the other case
is the same.  If 
$t=\f{x} \cap \f{z}$, then $\{\f{x},\f{z}\} \subset I_t$, so by (\ref{e:suslin})
we have $t \in T$. Clearly, \[\bar{z} \in \mc{S}\cap
I_{\f{z}|_{\delta_{\f{x}}+1}}=\mc{S} \cap I_{t \concat \f{z}(\delta_{\f{x}})}\] hence, by the
definition of $D$ we obtain that there exists a $\bar{p} \in D \cap I_{t \concat
\f{z}(\delta_{\f{x}})}.$  We have $\bar{p}|_{\delta_{\f{x}}+1}=\f{z}|_{\delta_{\f{x}}+1}$ so from 
$\delta_{\bar{x}} \geq \delta_{\f{y}}$ we get
\[\delta(\f{x},\f{p})=\delta_{\bar{x}}  \text{ and
}\delta(\f{y},\f{p})=\delta_{\bar{y}},\]
moreover \[\f{p}(\delta_{\bar{x}})=\f{z}(\delta_{\bar{x}}) \text{ and }
\f{p}(\delta_{\bar{y}})=\f{z}(\delta_{\bar{y}}).\]

Therefore, using (\ref{e:xz}) we obtain that
$\bar{x}<_{altlex}\f{p}<_{altlex}\f{y}$, so $\bar{p} \in D \cap
(\bar{x},\bar{y}) \subset D \cap J$. So $D$ is a countable dense subset of $\mc{S}$, a contradiction. 

This yields that $T$ is uncountable, hence it is indeed a Suslin tree.

Finally, notice that $T$ is a subtree of $\sigma^*[0,1]$ so $T|_{succ} \subset \sigma^*[0,1]|_{succ}$. Let $T'=T|_{succ}$. Clearly, $T'$ is a subset of $T$ and by definition the ordering of $T'$ is the restriction of the ordering of $T$, so $T'$ does not contain uncountable chains or antichains. 
In order to see that $T'$ is
uncountable first notice that the lengths of the elements in $T$ are unbounded in
$\omega_1$, therefore the lengths of the elements on the successor levels are also
unbounded. Hence $T'$ is uncountable so $T'$ is also a Suslin tree, which completes the proof of the lemma.
\end{proof}
For the sake of completeness we will prove the following classical facts about
Suslin trees.

\begin{lemma}
\label{l:suslindense}
If $D$ is a dense open subset of the Suslin tree $T$ then $T \setminus D$ is countable.
\end{lemma}
\begin{proof}
Let $A$ be a maximal antichain in $D$. Clearly, $A$ is countable. Let $\alpha$
be such that $\alpha>\sup \{l(s):s \in A\}$. Now, if $\beta \geq \alpha$ arbitrary
and $t \in Lev_{\beta}(T)$ then by the density of $D$ there exists an $s_0 \in
D$ such that $t \leq_T s_0$. From the facts that $A$ is maximal and
$\beta\geq\alpha$ we obtain that for some $s_1 \in A$ we have $s_1 \leq_T s_0$ and
hence $s_1 \leq_T t$. But then, as $D$ is open and $A \subset D$ we obtain that
$t \in D$. This finishes the proof of the lemma.
\end{proof}
\begin{lemma}
\label{l:suslinrspec}
A Suslin tree is not $\R$-special.
\end{lemma}
\begin{proof}
Suppose the contrary. Let $T$ be a Suslin tree and $f:T \to \R$ be an order preserving map. We can suppose that $f(T)$ is a subset of $[0,1]$.

 Let $n \in \omega$ and \[D_n=\{t \in T: (\forall s \geq_T
t)(f(s) \leq f(t)+\frac{1}{n+1})\}.\] Clearly, $D_n$ is open. We will show that
it is also dense in $T$. In order to see this let $t_0 \in T$ be arbitrary. Then
either $t_0 \in D_n$ or there exists an $t_1 \geq_T t_0$ such that
$f(t_1)>f(t_0)+\frac{1}{n+1}$. Repeating this argument for $t_1$ we obtain
either that $t_1 \in D_n$ or a $t_2 \geq_T t_1 $ such that
$f(t_2)>f(t_1)+\frac{1}{n+1}>f(t_0)+\frac{2}{n+1}$, etc.
$f(T) \subset [0,1]$ implies that this procedure stops after at most $n+2$
steps, hence we obtain an $s \geq_T t_0$ such that $s \in D_n$. Therefore, the
sets $D_n$ are dense open subsets of $T$. By Lemma \ref{l:suslindense}
the complement of $\cap_{n \in \omega}D_n$ is countable, hence there exists $s<_Tt$ such that $s,t \in \cap_{n \in \omega D_n}$. But then clearly $f(t)=f(s)$, a contradiction. 
\end{proof}
Now we are ready to prove the main result of this subsection.
\begin{proof}[Proof of Theorem \ref{t:suslin}]
Suppose the contrary and let $\mc{S}'$ be a subset of $\mc{B}_1$ order
isomorphic to a Suslin line. By the Main Theorem there exists an embedding
$\Phi_0:\mathcal{S}' \hookrightarrow [0,1]^{<\omega_1}_{\searrow 0}$. For $p,q \in
\mathcal{S}'$ let $p \sim q$ if the interval $[p,q]$ is separable. Then  $\sim$
is an equivalence relation and $\mc{S}=\mathcal{S}'/\sim$ is a nowhere separable
Suslin line  (for the details see \cite[Section 3., p. 252]{tod}). For every $\sim$ equivalence
class $[\cdot]$ fix a representative $p \in \mc{S}'$. It is easy to see that
every equivalence class is a convex set, so the map $\Phi([p])=\Phi_0(p)$ is an
order preserving embedding of $\mc{S}$ into $[0,1]^{<\omega_1}_{\searrow 0}$.

Now we can use Lemma \ref{l:suslin} for $\Phi(\mc{S})$. This yields that there
exists a Suslin tree $T \subset \sigma^*[0,1]|_{succ}$. Assign to each $t \in T$ the
last element of $t$, namely, let $f(t)=t(l(t)-1).$

Let $s,t \in T$ such that $s <_T t$. Then, as $s \not =t$, the sequences $s$ and
$t$ are strictly decreasing and (using that $s<_T t \iff s \subsetneqq t$) $t$
is an end extension of $s$ we obtain that $f(t)<f(s)$. Therefore, the map $1-f$
is a strictly monotone map from the Suslin tree $T$ to $\R$. This contradicts
Lemma \ref{l:suslinrspec}.
\end{proof}

\subsection{Linearly ordered sets of cardinality \texorpdfstring{$<\mathfrak{c}$}{c} and Martin's Axiom}

In this subsection we reprove the results of Elekes and Stepr\=ans from
\cite{marci2}. To formulate the statements, we need some preparation.

Suppose that $(L,<_L)$ is a linearly ordered set. A partition tree $T_L$ of $L$
is defined as follows:
the elements of $T_L$ are certain nonempty convex sets of $L$ ordered by
reverse inclusion. $T_L$ is constructed by induction. Let $Lev_0(T_L)=\{L\}$. 

Suppose that for an ordinal $\alpha$ we have defined $Lev_\beta(T_L)$ for all
$\beta<\alpha$. If $\alpha$ is a successor, for every $I \in Lev_{\alpha-1}(T_L)$ 
fix convex sets $I_0$ and $I_1$  such that $I_0 \cup I_1 =I$ and $I_0
\cap I_1=\emptyset$ if such $I_0,I_1$ exist. Let
\[Lev_{\alpha}(T_L)=\bigcup\{I_0,I_1: I \in Lev_{\alpha-1}(T_L) \}.\]
Now if $\alpha$ is a limit ordinal let 
\[Lev_{\alpha}(T_L)=\{\bigcap_{\beta<\alpha} I_\beta: I_\beta \in
Lev_{\beta}(T_L),\cap_{\beta<\alpha} I_\beta \not = \emptyset\}.\]
Somewhat ambiguously if $t \in T_L$ we will denote the corresponding convex set of $L$ by $N_t$. 

We first verify the next proposition, which is interesting in its
own right.
  
 \begin{proposition}
 \label{p:rspec}
  Let $L$ be a linear ordering such that $T_L$, a partition tree of $L$, is
$\mathbb{R}$-special. Then $L \hookrightarrow \mathcal{B}_1.$
 \end{proposition}
\begin{proof}
 Without loss of generality we can suppose that we have a strictly decreasing map $\Phi:
T_L \to (0,1)$. 
 
 \begin{lemma}
 \label{l:phidef}
 There exists a map $\Psi_0:T_L \to \sigma^*[0,1]$ with the following properties for
every $t,s \in T_L$:
 \begin{enumerate}[(1)]
  \item if $s \leq_{T_L} t$ then $\Psi_0(s) \subset \Psi_0(t)$,
  \item if $N_s<_L N_t$ then  $\Psi_0(s) <_{altlex} \Psi_0(t)$,
  \item $\inf \Psi_0(t)  \geq \Phi(t)$.
 \end{enumerate}
 \end{lemma}
 \begin{proof}
We define $\Psi_0$ inductively on the levels of $T_L$. Suppose that we are done
for every $\beta<\alpha$.

If $\alpha$ is a limit ordinal and $t \in Lev_{\alpha}(T_L)$, let 
\begin{equation} 
\label{e:rspeclim}
\Psi_0(t)=\bigcup_{t'  <_{T_L} t} \Psi_0(t').
\end{equation}
Now let $\alpha$ be a successor ordinal. 
First notice that for every $t \in Lev_{\alpha}(T_L)$ by the fact that $\Phi$ is strictly decreasing and the inductive hypothesis for $t|_\alpha$ we have
\begin{equation} 
 \label{e:rspecnew}
 \Phi(t)<\Phi(t|_{\alpha})  \leq \inf \Psi_0(t|_{\alpha}).
\end{equation}
Let \[A=\{t \in Lev_\alpha(T_L):(\exists s \in Lev_\alpha(T_L))(s \not = t \land t|_{\alpha}=s|_{\alpha})\}.\]
Now, if $t \not \in A$ then using (\ref{e:rspecnew}) there exists an $r \in [0,1]$ such that  
\begin{equation}
\label{e:rspecsuc1}
\Phi(t)<r< \inf \Psi_0(t|_\alpha).
\end{equation}
So let 
\begin{equation}
\label{e:rspecsuc2}
\Psi_0(t)=\Psi_0(t|_{\alpha}) \concat r.
\end{equation}
Notice that if $t \in A$ then there exists exactly one $s \not = t$ such that $s \in
Lev_{\alpha}(T_L)$ and $t|_{\alpha}=s|_{\alpha}$. Hence $A$ is the union of pairs $\{s,t\}$ such that $s,t \in Lev_{\alpha}(T_L)$ and $t \not = s$ and
$t|_\alpha=s|_\alpha$. We will define $\Psi_0(s)$ and $\Psi_0(t)$ simultaneously for such pairs. Since $s$ and $t$ are incomparable, the sets $N_s$ and $N_{t}$ are disjoint, so
either $N_{s}<_L N_{t}$ or $N_{s}>_L N_{t}$.
Using (\ref{e:rspecnew}) and $s|_\alpha=t|_\alpha$ we obtain 
\[
 \Phi(t),\Phi(s) <\Phi(t|_{\alpha})  \leq \inf \Psi_0(t|_{\alpha}).
\]
 From this it follows that we
can choose $r, q \in (0,1)$ such that 
\begin{equation}
\label{e:rspec2}
\Phi(t),\Phi(s) <r,q< \inf \Psi_0(t|_\alpha)
\end{equation}
and 
\begin{equation}
\label{e:rspec3}
N_{s} <_L N_{t} \iff \Psi_0(t|_{\alpha}) \concat q <_{altlex}
\Psi_0(t|_{\alpha}) \concat r, 
\end{equation}
so let
\begin{equation}
\label{e:rspec4}
\Psi_0(t)=\Psi_0(t|_{\alpha}) \concat r \text{ and }
\Psi_0(s)=\Psi_0(t|_{\alpha}) \concat q=\Psi_0(s|_{\alpha}) \concat q. 
\end{equation}
Thus, we have defined $\Psi_0$ on $Lev_{\alpha}(T_L)$ (first on the complement of $A$ then on $A$ as well). We claim that $\Psi_0$
satisfies properties (1)-(3). 

We check (1). Let $s <_{T_L} t$ and $t \in Lev_{\alpha}(T_L)$. If $\alpha$ is a limit ordinal then
by (\ref{e:rspeclim}) clearly $\Psi_0(s) \subset \Psi_0(t)$. If $\alpha$ is a
successor then $s \leq_{T_L} t|_{\alpha}$, hence from the inductive hypothesis and
from equations (\ref{e:rspecsuc1}) and (\ref{e:rspec4}) we obtain (1).

In order to prove (2) let $s$ and $t$ be given with $N_s<_L N_t$. It follows from the definition of the partition tree that $s$ and $t$ are $\leq_{T_L}$ incomparable. Let $\beta$ be the minimal ordinal such that there exist $s',t' \in Lev_{\beta}(T_L)$ with $s' \leq_{T_L} s$, $t' \leq_{T_L} t$ and $s'$ and $t'$ are still incomparable. Suppose first that $\beta=\alpha$. Then, of course $s'=s,$ $t'=t$ and by the definition of a partition tree we have that $\alpha$ cannot be a limit ordinal. Hence, by equations (\ref{e:rspec3}) and (\ref{e:rspec4}) clearly (2)
holds. Now, if $\beta<\alpha$ then $N_{s'}<N_{t'}$. Hence from the inductive hypothesis $\Psi_0(s')<_{altlex}
\Psi_0(t')$ so from property (1) we have $\Psi_0(s)<_{altlex} \Psi_0(t)$.

Finally, in order to see (3) if $\alpha$ is a limit just notice that $\Phi(t) \leq_L
\Phi(t')$ whenever $t' \leq_{T_L} t$ so by the inductive hypothesis we have
\[\Phi(t) \leq \inf_{t'  <_{T_L} t} \Phi(t') \leq \inf_{t'   <_{T_L} t}(\inf
\Psi_0(t'))= \inf \Psi_0(t).
\]
If $\alpha$ is a successor then for $t \not \in A$ by (\ref{e:rspecsuc1})
and (\ref{e:rspecsuc2}), while for $t \in A$ by (\ref{e:rspec2}) and (\ref{e:rspec4}) we get
(3).

Thus the induction works, so we have proved that such a $\Psi_0$ exists.
\end{proof}
Now we define the embedding $L \hookrightarrow [0,1]^{<\omega_1}_{\searrow 0}$. For $x \in
L$ let \[\Psi(x)=(\bigcup_{t\in T_L,\text{ }x \in N_t} \Psi_0(t)) \concat 0.\]
By the definition of a partition tree, if for $s$ and $t$ we have $x \in N_t \cap
N_s$ then $s$ and $t$ are $\leq_{T_L}$-comparable. Hence by property (1) of
$\Psi_0$ for every $x \in L$ we have $\Psi_0(x) \in \sigma^*[0,1]$. Moreover, since
$ran(\Phi) \subset (0,1)$ and by property $(3)$ we have that concatenating
$\bigcup_{t\in T_L,\text{ }x \in N_t} \Psi_0(t)$ with zero will give an element in
$\uni$. 

We claim that the map $\Psi$ is order preserving between $(L,<_L)$ and $\un$.
Let $x,y \in L$ with $x<_L y$. Then there exist  $s,t \in T_L$ such that
$x \in N_{s}$ and $y \in N_{t}$ and $N_{s} <_L N_{t}$. Then by property (2) of
$\Psi_0$ we have $\Psi_0(s) <_{altlex}  \Psi_0(t)$. Therefore,
$\Psi_0(s) \subset \Psi(x)$ and $\Psi_0(t) \subset \Psi(y)$ implies
$\Psi(x)<_{altlex} \Psi(y)$.
\end{proof}
\begin{theorem}
 \label{t:MAordembed}
 (MA) If $L$ is a linearly ordered set of cardinality $<\mathfrak{c}$ then $L$
is representable in $\mathcal{B}_1$ iff $L$ does not contain $\omega_1$ or
$\omega^*_1$. 
\end{theorem}
\begin{proof} Suppose that $L$ does not contain $\omega_1$ or $\omega^*_1$. Let $T_L$ be a partition tree of $L$. We claim that $T_L$ does not contain
uncountable chains. Suppose the contrary, let $\{t_\alpha:\alpha<\omega_1\}
\subset T_L$ be a chain. Then $N_{t_\alpha}$ (denoted by $N_\alpha$ later on) is a strictly decreasing sequence
of convex  sets in $L$. Therefore, for every $\alpha$ there exists an $x_\alpha \in
N_{\alpha} \setminus N_{\alpha+1}$ such that either $N_{\alpha+1}<_L \{x_\alpha\}$
or $N_{\alpha+1}>_L \{x_\alpha\}$. Without loss of generality we can suppose that the set
$R=\{\alpha:(\exists x_\alpha \in N_{\alpha} \setminus N_{\alpha+1})(N_{\alpha+1}<_L
\{x_\alpha\})\}$ is uncountable.
But then the sequence $(x_\alpha)_{\alpha \in R}$ is strictly decreasing in $L$ and $R$ is unbounded in $\omega_1$ so $(x_\alpha)_{\alpha \in R}$ is order isomorphic to
$\omega^*_1$.

Notice that as every level of $T_L$ contains pairwise disjoint nonempty convex sets of $L$, from $|L|<\mathfrak{c}$ it follows that the cardinality of every level is strictly less than $\mathfrak{c}$. Moreover, since $T_L$ does not contain uncountable chains, using that under Martin's Axiom $\mathfrak{c}$ is a regular cardinal we obtain that $|T_L|<\mathfrak{c}$. 

Now it is easy to prove the theorem using a result of Baumgartner, Malitz and
Reinhardt (see \cite{baumg}) which states that assuming Martin's Axiom every
tree with cardinality $<\mathfrak{c}$ that does not contain $\omega_1$-chains is
$\mathbb{Q}$-special. 
We have seen that $T_L$ does not contain uncountable chains and $|T_L|<\mathfrak{c}$, hence it is $\mathbb{Q}$-special (in particular $\R$-special), so by
Proposition \ref{p:rspec} we have $L \hookrightarrow \uni$. By the Main Theorem this
implies $L \hookrightarrow \mathcal{B}_1$. 
\end{proof}

\section{New results}
\label{s:new}

\subsection{Countable products and gluing}

In this section we will answer Questions 2.2, 2.5 and 3.10 from \cite{marci1}.
Concerning the last question we would like to point out that in fact it has been
already solved in \cite{marci2}.

Elekes \cite{marci1} investigated several operations on collections of
linearly ordered sets, and asked  whether the closure of a simple collection
of orderings under these operations coincide with the linearly ordered subsets
of $\mathcal{B}_1$. We will first prove that the set of linearly ordered subsets
of $\mc{B}_1$ is closed under the application of these operations. 

\begin{definition}
Let $L$ be a linearly ordered set and for every $p \in L$ fix a  linearly
ordered set $L_p$. 
Then the set $\{(p,q):p \in L, q \in L_p\}$ ordered lexicographically (that is, $(p,q)<_g(p',q')$ if
and only if $p<_{L}p'$ or $p=p'$ and $q<_{L_p}q'$) is called the \emph{gluing of the
$L_p$'s along $L$}. 
\end{definition}

\begin{theorem}
\label{t:productglue}
\begin{enumerate}[(1)]
\item Let $\{L_{\beta}:\beta<\alpha\}$ be a countable collection of linearly
ordered sets that are representable in $\mathcal{B}_1$. Then the set
$\prod_{\beta<\alpha} L_{\beta}$ ordered lexicographically is also
representable. 
\item Suppose that $L$ and every $(L_p)_{p \in L}$ is representable in
$\mc{B}_1$. Then the gluing of $L_p$'s along $L$ is also representable in
$\mc{B}_1$.
\end{enumerate}
\end{theorem}
{\sc Notation.} Throughout this section if $\bar{x}=(x_\alpha)_{\alpha \leq
\xi}$ is a transfinite sequence of reals and $a,b \in \R$ we will abbreviate the
sequence $(ax_\alpha+b)_{\alpha \leq \xi}$ by $a\bar{x}+b$. 

First we need a technical lemma.
\begin{lemma}
\label{l:evenlength}
Suppose that $L$ is a linearly ordered set and there exists an embedding $\Psi:L
\hookrightarrow [0,1]^{<\omega_1}_{\searrow 0}$. Then there exists an embedding $\Psi':L
\hookrightarrow [0,1]^{<\omega_1}_{\searrow 0}$ such that for every $p \in L$ the length
$l(\Psi'(p))$ is an even ordinal. 
\end{lemma}
\begin{proof}
It is easy to see that 
\[\Psi'(p)=\begin{cases}
                 ( \frac{1}{2} \Psi(p)+\frac{1}{2})\concat 0 & \text{if }
l(\Psi(p)) \text{ is odd}\\
                 ( \frac{1}{2} \Psi(p)+\frac{1}{2})\concat \frac{1}{4}\concat 0
& \text{if } 
                 l(\Psi(p)) \text{ is even}\\
                 \end{cases}
\]
is also order preserving and takes every point $p \in L$ to a sequence with even
length.
\end{proof}
\begin{proof}[Proof of Theorem \ref{t:productglue}]
 
 First we prove $(1)$. The representability of $L_\beta$ for every
$\beta<\alpha$ by the Main Theorem imply that there exist embeddings
$\Psi_\beta: L_\beta \hookrightarrow \uni$. Using Lemma \ref{l:evenlength} we
can suppose that for every $\beta<\alpha$ and $p \in L_\beta$ the length of
$\Psi_\beta(p)$ is even. 
 
 Fix now a sequence $(y_\beta)_{\beta \leq \alpha} \in \sigma^*[\frac{1}{2},1]$.
For $\bar{p}=(p_\beta)_{\beta<\alpha}  \in \prod_{\beta<\alpha} L_\beta$ let
 \[\Psi(\f{p})=(\concatt_{\beta<\alpha}
(\frac{y_{\beta}-y_{\beta+1}}{2}\Psi_\beta(p_\beta)+y_{\beta+1})) \concat 0,\]
where $\concatt_{\beta<\alpha}$ denotes concatenation of the sequences in type
$\alpha$.

We claim that $\Psi$ is an embedding of $(\prod_{\beta<\alpha}
L_\beta,<_{lex})$ into $\un$. 
 It is easy to see that for every $\f{p} \in \prod_{\beta<\alpha} L_\beta$ we have $\Psi(\f{p}) \in
\uni$.
 
Now we prove that $\Psi$ is order preserving. Let $\bar{p} <_{lex} \f{q}$ with
$\bar{p}=(p_\beta)_{\beta<\alpha}$, $\bar{q}=(q_\beta)_{\beta<\alpha}$ and let 
$\delta=\delta(\f{p},\f{q})$, then $p_\delta<_{L_\delta} q_\delta$. It is easy to see that 
\[
 \delta(\Psi(\f{p}),\Psi(\f{q}))=\sum_{\beta<\delta} l(\Psi_\beta(p_\beta))+\delta(\Psi_\delta(p_\delta),\Psi_\delta(q_\delta)).
\]
In particular, since every length in the previous equation is even we get that
the $\delta(\Psi(\f{p}),\Psi(\f{q}))$ and $\delta(\Psi_\delta(p_\delta),\Psi_\delta(q_\delta))$ are of the same parity. Using this, $p_\delta<_{L_\delta} q_\delta$ and the fact that $\Psi_\delta$ is order preserving, we obtain that $\Psi(\bar{p})<_{altlex}\Psi(\bar{q})$, which finishes the proof of (1).

(2) can be proved similarly. Fix an order preserving
embedding $\Psi_0:L \hookrightarrow \uni$ such that for every $p \in L$ we have
that $l(\Psi(p))$ is even. For every $p \in L$ let us also fix embeddings
$\Psi_p: L_p \hookrightarrow \uni$. Then 
\[\Psi(p,q)=(\frac{1}{2}(\Psi_0(p))+\frac{1}{2}) \concat
(\frac{1}{8}(\Psi_p(q))+\frac{1}{4}) \concat 0\]
works.
 \end{proof}

\begin{definition}
 Let $L$ be a linearly ordered set. The set $L
\times 2$ ordered lexicographically is called the \emph{duplication of $L$}. 
\end{definition}
\begin{corollary}
\label{c:duplication}
 A linearly ordered set is representable in $\mc{B}_1$ then its duplication is
also representable.
\end{corollary}
The first part of Theorem \ref{t:productglue} answers Question 2.5, while
Corollary \ref{c:duplication} answers Question 2.2 from \cite{marci1}
affirmatively. 

Now let us define the above mentioned operations on collections of linearly ordered
sets. Suppose that $\mathcal{H}$ is an arbitrary set of ordered sets. 
\begin{definition}
Let $\alpha<\omega_1$ be an ordinal, then \[\mc{H}^{\alpha}=\{L_1 \subset
L^{\alpha}:L \in \mathcal{H}\},\]
where $L^\alpha$ is ordered lexicographically.
Let us denote by $\mathcal{H}^*$ the closure of $\mc{H}$ under the operation
$\mc{H} \mapsto \mc{H}^{\alpha}$ for every $\alpha<\omega_1$.
\end{definition}

\begin{definition}
$\mathcal{S}(\mc{H})$ denotes the closure of $\mathcal{H}$ under gluing.
 
\end{definition}
It can be shown that such $\mc{H}^*$ and $\mc{S}(\mc{H})$ exist.

Suppose that every element of $\mathcal{H}$ is representable in $\mc{B}_1$. The
first part of Theorem \ref{t:productglue} clearly implies that every element of 
$\mathcal{H}^*$, while the second part yields that every element of
$\mc{S}(\mc{H})$ is representable in $\mc{B}_1$. So it is natural to ask the
following:

\begin{question} (Elekes, \cite[Question 3.10.]{marci1}) \label{q:marci1}Does
$S(\{[0,1]^\alpha:\alpha<\omega_1\})^\omega$ or $S(\{[0,1]^\alpha:\alpha<\omega_1\})^*$
equal to the linearly ordered sets representable in $\mc{B}_1$?
\end{question}
To answer this question we need a property that is invariant under the above operations.
\begin{definition}
 We will say that a linearly ordered set $L$ has property (*) if every uncountable subset of $L$ contains an uncountable subset order-isomorphic to a subset of $\R$. 
\end{definition}

\begin{proposition}
\label{p:propstar}
Suppose that every $L \in \mc{H}$ has property (*). Then (*) holds for every
element of $\mc{H}^*$ and $\mc{S}(\mc{H})$  as well.
\end{proposition}

\begin{proof}

In order to prove that every element of $\mc{H}^*$ has the required property it is enough
to prove that if $\alpha<\omega_1$ and $L$ has property (*) then so does
$L^\alpha$.

We prove this by induction on $\alpha$. Suppose that we are done for every
$\beta<\alpha$ and let $L_1 \subset L^\alpha$ be uncountable.

Observe that if there exists an ordinal $\beta <\alpha $ such that $L_2=\{\bar{p}
\in L^\beta:(\exists \bar{q})(\bar{p}\concat\bar{q}\in L_1)\} $ is uncountable then
using that $L_2 \subset L^\beta$ and the inductive hypothesis we obtain that
$L_2$ contains an uncountable real order type $R_2$. Thus, there exists an $R_1
\subset L_1$ such that for every $\bar{p}\in R_2$ there exists a unique $\bar{q}_{\bar{p}}$
such that $R_1=\{(\bar{p},\bar{q}_{\bar{p}}):\bar{p} \in R_2\}$. It is easy to see that since $L^\alpha$ is
ordered lexicographically we have that $R_1$ is an uncountable real order type
in $L_1$ (in fact it is isomorphic to $R_2$).

So we can suppose  that there is no such a $\beta$.

If $\alpha$ is a successor then using the above observation for $\beta=\alpha-1$ we
obtain that the set $\{\bar{p} \in L^{\alpha-1}:(\exists q \in L)(\bar{p}\concat q\in
L_1)\}$ is countable. By the uncountability of $L_1$ there exists a $\bar{p} \in
L^{\alpha-1}$ such that the set $\{q:\bar{p}\concat q \in L_1\}$ is uncountable. But
this is a subset of $L$, so by the assumption on $L$ there exists an
uncountable real order type  $R \subset \{q:\bar{p}\concat q \in L_1\}$. Then
$\{\bar{p}\concat q:q \in R\}$ is an uncountable real order type in $L_1$.

Suppose now that $\alpha$ is a limit ordinal. By the above observation for every
$\beta<\alpha$ the set $\{\bar{p} \in L^\beta:(\exists
\bar{q})(\bar{p}\concat\bar{q}\in L_1)\}$ is countable. So there exist countable
sets $D_\beta \subset L_1$ with the following property: whenever for a point
$\bar{p} \in L^\beta$ there exists a $\bar{q}$ such that $\bar{p}\concat\bar{q} \in
L_1$ then there exists a $\bar{q}'$ such that $\bar{p}\concat\bar{q}'\in D_\beta$.
Let $D=\bigcup_{\beta<\alpha} D_\beta$, then $D$ is a countable set.

We claim that $D$ is dense in $L_1$ (equipped with the order topology). In order to prove this
let $\f{x},\f{y} \in L_1$ such that $(\f{x},\f{y}) \cap L_1$ is nonempty. Choose a
$\f{z} \in (\f{x},\f{y}) \cap L_1$. Since $\alpha$ is a limit there exists a
$\beta<\alpha$ such that $\beta>\max\{\delta(\f{x},\f{z}),\delta(\f{y}, \f{z})\}$.
Then there exists a $\f{w} \in D_\beta \subset D$ such that
$\f{w}|_\beta=\f{z}|_\beta$. But then clearly $\f{w} \in (\f{x},\f{y}) \cap L_1
\cap D$. So $D$ is indeed dense.
Consequently, $L_1$ contains an uncountable real order type (see \cite[3.2.
Corollary]{tod}). This proves that $L^\alpha$ has property (*), so it is true
for every element of $\mc{H}^*$.

In order to prove that every element of $\mc{S}(\mc{H})$ has property (*) one can use
similar ideas: just use the above observation and the same argument as in the case of
successor $\alpha$.\end{proof}
 
Now we are ready to answer Question \ref{q:marci1}. 
 An \emph{Aronszajn line} is an uncountable linearly ordered set that does not contain $\omega_1$, $\omega^*_1$ and uncountable sets isomorphic to a subset of $\R$. An Aronszajn line is called \textit{special} if it has an $\R$-special partition tree. Special Aronszajn lines exist, see \cite[Theorem 5.1, 5.2]{tod}.
Notice that Proposition
\ref{p:rspec} immediately gives the following important corollary:
\begin{corollary}
If $A$ is a special Aronszajn line then $A \hookrightarrow \mathcal{B}_1$.
\end{corollary}

This corollary was proved by Elekes and Stepr\=ans. Although it is not mentioned
explicitly in the Elekes-Stepr\=ans paper, the embeddability of the Aronszajn line
answers the questions of Elekes negatively: on the one hand an Aronszajn line
does not contain uncountable real order types. On the other hand by Proposition
\ref{p:propstar} every element of every collection of linear orderings
obtainable from $\{[0,1]\}$ by the operations $\mathcal{H} \mapsto \mc{H}^*$ or
$\mathcal{H} \mapsto \mathcal{S}(\mathcal{H})$ has property (*).  
\subsection{Completion}
Now we will answer Question 2.7 from \cite{marci1} negatively.
\begin{theorem}
\label{t:complete}
There exists a linearly ordered set such that it is representable in
$\mathcal{B}_1$, but none of its completions are representable.
\end{theorem}
\begin{proof}
Let $L \supset [0,1]^{<\omega_1}_{\searrow 0}$ be a completion of $[0,1]^{<\omega_1}_{\searrow 0}$, that is, a complete linear order containing $[0,1]^{<\omega_1}_{\searrow 0}$ as a dense subset. If
it was representable then by Corollary \ref{c:duplication} there would be an order
preserving embedding $\Psi:L \times 2 \hookrightarrow [0,1]^{<\omega_1}_{\searrow 0}$. We will denote the lexicographical ordering on $L \times 2$ by $<_{L \times 2}$ and somewhat ambiguously the lexicographical ordering on $[0,1]^{<\omega_1}_{\searrow 0} \times 2$ by $<_{altlex \times 2}$. Notice that $<_{altlex \times 2}$ is the restriction of $<_{L \times 2}$ to $[0,1]^{<\omega_1}_{\searrow 0} \times 2$.

{\sc Notation.} For each $s \in \sigma^*[0,1]$ let $J_s$ be the set $\{\bar{x} \in
[0,1]^{<\omega_1}_{\searrow 0}: s \subset \bar{x}\} \times 2$. We
will use the notation 
\begin{equation}
\label{e:compsdef}
I(s)=\Psi(\inf(J_s)) \text{ and }S(s)=\Psi(\sup(J_s)).
\end{equation}
Notice that if $L$ is complete then the set $L \times 2$ ordered lexicographically is also a complete linearly
ordered set, hence $I(s)$ and $S(s)$ exist for every $s \in \sigma^*[0,1]$.

Let us define a map $\Phi:\sigma^* [0,1] \to [0,1]$ as follows: 

\begin{definition}

 For $s \in \sigma^*[0,1]$ let \[\delta_s=\delta(I(s),S(s))\]
 and 
\[\Phi(s)=\max\{I(s)(\delta_s),S(s)(\delta_s)\}.\]
\end{definition}

Let us also use the notation \[\phi(s)=\min\{I(s)(\delta_s),S(s)(\delta_s)\}.\]
Notice that $\Phi$ and $\phi$ are well defined, since for every $s \in
\sigma^* [0,1]$ the set $J_s$ contains at least two elements (one with last element $0$
and another with $1$), so $I(s)$ and $S(s)$ must differ. From this we have for all $s$ that
\begin{equation}
\label{e:defphitriv}
0\leq \phi(s)<\Phi(s).
\end{equation}
In the following lemma we collect the easy observations that will be needed in the proof of the theorem.
\begin{lemma}
 \label{l:obs}
 Let $s,t,u \in \sigma^*[0,1]$ with $s \subset t$. Then 
 \begin{enumerate}[(1)]
 \item $\delta_s \leq \delta_t$,
 \item 
 \begin{enumerate}
 \item $\Phi(s) \geq \Phi(t)$,
 \item $\max\{I(t)(\delta_s),S(t)(\delta_s)\} \leq \Phi(s)$,
 
 \end{enumerate}
 \item if $\delta \leq \delta_t$ then $\Phi(t)\leq \max\{I(t)(\delta),S(t)(\delta)\}$,

 \item if $\Phi(s)=\Phi(t)$ then $\delta_s=\delta_t$,
 
 \item if $r,q \in [0,1]$ such that $t\concat r \leq_{altlex} t \concat q$ then 
 \begin{enumerate}
  \item $I(t \concat r)|_{\delta_t} = S(t \concat r)|_{\delta_t}= I(t \concat q)|_{\delta_t}= S(t \concat q)|_{\delta_t}$,
 
 \item $I(t \concat r)(\delta_t) \leq S(t \concat r)(\delta_t)\leq I(t \concat q)(\delta_t)\leq S(t \concat q)(\delta_t)$ if $\delta_t$ is even,
 \item $I(t \concat r)(\delta_t) \geq S(t \concat r)(\delta_t)\geq I(t \concat q)(\delta_t)\geq S(t \concat q)(\delta_t)$  if $\delta_t$ is odd,
\end{enumerate}
 \item if $t\leq_{altlex}u$ and $\delta$ is an even ordinal such that $I(t)|_{\delta}=S(t)|_{\delta}=I(u)|_{\delta}$ then
  \[I(t)(\delta)\leq S(t)(\delta)  \leq I(u)(\delta).\]
  \end{enumerate}
\end{lemma}
\begin{proof}
$J_s \supset J_t$, so by the fact that $\Psi$ is order preserving we get \[I(s) \leq_{altlex} I(t) \leq_{altlex} S(t)\leq_{altlex}S(s).\]
Therefore, by the definition of $<_{altlex}$ it is clear that $\delta_s \leq \delta_t$, so we have (1).

Now we show part (b) of (2). It is easy to see from the definition of $<_{altlex}$ that for every $\bar{x} \in [\inf(J_s),\sup(J_s)]$ we have $\Psi(\bar{x})(\delta_s) \in
[\phi(s),\Phi(s)]$. In particular, as $[\inf(J_t),\sup(J_t)] \subset [\inf(J_s),\sup(J_s)]$ we obtain 
\begin{equation}
\label{e:compest}
\max\{I(t)(\delta_s),S(t)(\delta_s)\} \in 
[\phi(s),\Phi(s)],
\end{equation}
which gives  part (b).
Since $I(t)$ and $S(t)$ are strictly decreasing sequences, using (1) we have
\[I(t)(\delta_t)  \leq I(t)(\delta_s) \text { and } S(t)(\delta_t)  \leq S(t)(\delta_s).\]
Hence, (\ref{e:compest}) yields that $\Phi(t) \leq \Phi(s)$. Thus we have verified (2).

In order to see (3), use again that the sequences $I(t)$ and $S(t)$ are  decreasing. Hence from $\delta\leq \delta_t$ and the definition of $\delta_t$ we have (3):
\[\Phi(t)=\max\{I(t)(\delta_t),S(t)(\delta_t)\}\leq \max\{I(t)(\delta),S(t)(\delta)\}.\]

In order to prove (4) using (1) it is enough to show that $\delta_s <\delta_t$ implies $\Phi(t)<\Phi(s)$. If $\delta_s <\delta_t$
then by the definition of $\delta_t$, the fact that the sequences $I(t)$ and $S(t)$ are strictly decreasing and \eqref{e:compest}, we obtain
\[\Phi(t)=\max\{I(t)(\delta_t),S(t)(\delta_t)\}< \max\{I(t)(\delta_s),S(t)(\delta_s)\} \leq \Phi(s),\]
which proves (4).

Now we prove (5). Notice that $t \concat r \leq_{altlex} t \concat q$ implies that $J_{t \concat r}\leq_{altlex \times 2} J_{t \concat q}$. Thus, \[\inf(J_{t \concat r}) \leq_{L \times 2}\sup(J_{t \concat r}) \leq_{L \times 2}\inf(J_{t \concat q}) \leq_{L \times 2}\sup(J_{t \concat q}).\]
Consequently, by the fact that $\Psi$ is order preserving, we get
\begin{equation}
\label{e:llong}
I(t \concat r)\leq_{altlex}S(t \concat r)\leq_{altlex}I(t \concat q)\leq_{altlex}S(t \concat q).
\end{equation}
From $J_{t \concat r}, J_{t \concat q} \subset J_t$ it is clear that
\[I(t)\leq_{altlex}I(t \concat r)\leq_{altlex}S(t \concat r)\]\[\leq_{altlex}I(t \concat q)\leq_{altlex}S(t \concat q) \leq_{altlex} S(t).\]
Thus, from the definition of $\delta_t$ we have 
\[I(t)|_{\delta_t}=I(t \concat r)|_{\delta_t}=S(t \concat r )|_{\delta_t}=I(t \concat q)|_{\delta_t}=S(t \concat q)|_{\delta_t}=S(t)|_{\delta_t},\]
so this shows that (a) holds. 
Now using (a), the definition of $<_{altlex}$ and \eqref{e:llong} we obtain (b) and (c) of (5) as well.

The proof of (6) is similar to the previous argument: $t\leq_{altlex} u$ implies $J_t \leq_{L \times 2} J_u$, consequently $I(t)\leq_{altlex}S(t)\leq_{altlex}I(u)$.
Since by assumption $\delta$ is even and $I(t)|_{\delta}=S(t)|_{\delta}=I(u)|_{\delta}$, the definition of $<_{altlex}$ implies
\[I(t)(\delta)\leq S(t)(\delta)  \leq I(u)(\delta).\]
\end{proof}

The following lemma is the essence of our proof. 
\begin{lemma}
\label{l:compmain}
There exists a $\subsetneqq$-increasing sequence $\{s_\alpha\}_{\alpha<\omega_1}$ such that $s_\alpha \in \sigma^*[0,1]$, $l(s_\alpha)=\alpha$ and  \[(\forall r \in s_\alpha)(\Phi(s_\alpha)<r). \tag{*}\]
\end{lemma}
\begin{proof}
We define $s_\alpha$ by induction on $\alpha$. 

Suppose that we have defined $s_\beta$ for $\beta<\alpha$. Then by the inductive hypothesis for every $\beta<\alpha$ we have \begin{equation}
\label{e:sbind0}
(\forall r \in
s_\beta)(\Phi(s_\beta)<r).
\end{equation}
Now we define $s_\alpha$ for limit and successor $\alpha$'s separately.

{\sc $\alpha$ is a limit.}  Let $s_\alpha=\bigcup_{\beta<\alpha} s_\beta$. If
$r \in s_\alpha$ is arbitrary then $r \in s_\beta$ for some $\beta<\alpha$. Notice that part (a) of (2) of Lemma \ref{l:obs} and \eqref{e:sbind0} imply
\[
(s_\beta \subset s_\alpha \text{ and } r \in s_\beta) \Rightarrow
\Phi(s_\alpha) \leq
\Phi(s_\beta)<r.
\] Hence, using $s_\beta \subset s_\alpha$ we obtain  $\Phi(s_\alpha)<r$ so $s_\alpha$ satisfies requirement (*).

{\sc $\alpha$ is a successor.} Let $\alpha=\beta+1$.

Our aim is to find a real $x$ such that 
\begin{equation}
\label{e:tosolve}
s_\beta \concat x \in \sigma^*[0,1] \text{ and } \Phi(s_\beta \concat x)<x. 
\end{equation}
 Clearly, this ensures that $s_\alpha=s_\beta \concat x$ satisfies (*).

Now notice that (\ref{e:sbind0}) yields \begin{equation}
\label{e:sbcc}s_\beta \concat \Phi(s_\beta) \in \sigma^*[0,1]. \end{equation} Now we have to separate two cases.

\textit{First}, suppose that \[
\Phi(s_\beta \concat \Phi(s_\beta))<\Phi(s_\beta).
\] Let $x=\Phi(s_\beta)$. It is clear that $x$ satisfies (\ref{e:tosolve}) by induction, so $s_\alpha=s_\beta \concat x$ is a suitable choice for (*).

\textit{Second}, suppose that $\Phi(s_\beta \concat \Phi(s_\beta))\geq \Phi(s_\beta).$
Since $s_\beta \subset
s_\beta \concat \Phi(s_\beta)$, by part (a) of (2) of Lemma \ref{l:obs} we have
$\Phi(s_\beta \concat \Phi(s_\beta)) \leq \Phi(s_\beta)$, so in fact 
\begin{equation}
 \label{e:yassumpt}
 \Phi(s_\beta \concat \Phi(s_\beta))=\Phi(s_\beta).
\end{equation}
Moreover, by (4) of Lemma \ref{l:obs} we obtain that (\ref{e:yassumpt}) implies 
\begin{equation}
\label{e:deltaequ}
\delta_{s_\beta \concat \Phi(s_\beta)}=\delta_{s_\beta}.
\end{equation} 
 In order to find an $x$ that satisfies (\ref{e:tosolve}) we will distinguish 3 cases according to the
parity of $\beta$ and $\delta_{s_\beta}$.

{\noindent \textbf{Case 1.}} $\beta$ and $\delta_{s_\beta}$ have the same parity.\\
By (\ref{e:defphitriv})  we can choose an \begin{equation}
\label{e:cxdef1}
x \in (\phi(s_\beta \concat \Phi(s_\beta)),\Phi(s_\beta \concat \Phi(s_\beta)))=(\phi(s_\beta \concat \Phi(s_\beta)),\Phi(s_\beta))
\end{equation}
where the equality holds because of (\ref{e:yassumpt}).

We claim that $x$ has property (\ref{e:tosolve}).
Clearly, $x<\Phi(s_\beta)$ and therefore by (\ref{e:sbind0}) we have $s_\beta \concat x \in \sigma^*[0,1]$, hence the first part of \eqref{e:tosolve} holds. Now we can use (5) of Lemma \ref{l:obs} (part (b) with $t=s_\beta$, $r=x$, $q=\Phi(s_\beta)$ if $\delta_{s_\beta}$ and $\beta$ are even and part (c) with $t=s_\beta$, $r=\Phi(s_\beta)$, $q=x$ if they are odd) and we obtain
 \begin{equation}
\label{e:isnotin}
\max\{I(s_\beta \concat x)(\delta_{s_\beta}),S(s_\beta \concat
x)(\delta_{s_\beta})\} \leq 
\end{equation}
\[
\min\{I(s_\beta\concat \Phi(s_\beta))(\delta_{s_\beta}),S(s_\beta\concat \Phi(s_\beta))(\delta_{s_\beta})\}=
\phi(s_\beta \concat \Phi(s_\beta))<x,
\]
where the equality follows from the definition of $\phi$ and \eqref{e:deltaequ} and the last inequality follows from \eqref{e:cxdef1}.

By (1) of Lemma \ref{l:obs} we have $\delta_{s_\beta} \leq \delta_{s_\beta \concat x}$ and (3) of Lemma \ref{l:obs} implies
\[\Phi(s_\beta \concat x)\leq \max\{I(s_\beta \concat
x)(\delta_{s_\beta}),S(s_\beta \concat x)(\delta_{s_\beta})\}.\]
Combining this inequality with (\ref{e:isnotin}) we obtain that the second part of (\ref{e:tosolve}) holds for $x$. So $s_\alpha=s_\beta \concat x$ satisfies (*), hence we are done with the first case.

{\noindent \textbf{Case 2.}} $\beta$ is even and $\delta_{s_\beta}$ is odd.\\ Then clearly, by \eqref{e:yassumpt}, (\ref{e:deltaequ}) and the odd parity of $\delta_{s_\beta}$
\[\Phi(s_\beta)=
\Phi(s_\beta
\concat \Phi(s_\beta))=\]
\[\max\{I(s_\beta \concat \Phi(s_\beta))(\delta_{s_\beta \concat \Phi(s_\beta)}),S(s_\beta \concat \Phi(s_\beta))(\delta_{s_\beta \concat \Phi(s_\beta)})\}=\]
\[\max\{I(s_\beta \concat \Phi(s_\beta))(\delta_{s_\beta}),S(s_\beta \concat \Phi(s_\beta))(\delta_{s_\beta})\}=I(s_\beta \concat \Phi(s_\beta))(\delta_{s_\beta}).
\]
Thus,
\begin{equation}\Phi(s_\beta)=I(s_\beta \concat \Phi(s_\beta))(\delta_{s_\beta}).
\label{e:compeven1}
\end{equation}
Let $z<\Phi(s_\beta)$ be arbitrary. Clearly, by the parity of $\beta$ we get $s_\beta \concat z
<_{altlex} s_\beta \concat \Phi(s_\beta)$.

Hence, using part (c) of (5) of Lemma \ref{l:obs} with $t=s_\beta$, $r=z$ and $q=\Phi(s_\beta)$ we obtain
\begin{equation}
\label{e:compnew0}I(s_\beta \concat z)(\delta_{s_\beta}) \geq S(s_\beta \concat z))(\delta_{s_\beta}) \geq I(s_\beta \concat
\Phi(s_\beta))(\delta_{s_\beta})
\geq S(s_\beta \concat \Phi(s_\beta))(\delta_{s_\beta}).
\end{equation}
Now, part (b) of (2) of Lemma \ref{l:obs} applied to $s_\beta $ and $s_\beta \concat z$ yields
\begin{equation}
\max\{I(s_\beta \concat
z)(\delta_{s_\beta}),S(s_\beta \concat
z)(\delta_{s_\beta})\} \leq \Phi(s_\beta).
\label{e:compnew1}
\end{equation}
Comparing this inequality with (\ref{e:compnew0}) and (\ref{e:compeven1}) we have 
\begin{equation}
\label{e:compnew2}
I(s_\beta \concat
z)(\delta_{s_\beta}) = S(s_\beta \concat
z)(\delta_{s_\beta})=I(s_\beta \concat \Phi(s_\beta))(\delta_{s_\beta}).
\end{equation}
Therefore, as by (1) of Lemma \ref{l:obs} $\delta_{s_\beta \concat z} \geq \delta_{s_\beta}$, we obtain that
\begin{equation}
\label{e:compeven2}
\text{ for every $z<\Phi(s_\beta)$ we have }
\delta_{s_\beta \concat z} \geq \delta_{s_\beta}+1. 
\end{equation}
Notice that (a) of (5) of Lemma \ref{l:obs} applied to $s_\beta \concat z$ and $s_\beta \concat \Phi(s_\beta)$ imply that
\begin{equation}
\label{e:compevennew}I(s_\beta \concat z) |_{\delta_{s_\beta}}=S(s_\beta \concat z) |_{\delta_{s_\beta}}=
I(s_\beta \concat \Phi(s_\beta)) |_{\delta_{s_\beta}}=S(s_\beta \concat \Phi(s_\beta)) |_{\delta_{s_\beta}}.
\end{equation}
Now the even parity of $\delta_{s_\beta}+1$, $s_\beta \concat z <_{altlex} s_\beta
\concat \Phi(s_\beta)$, (\ref{e:compnew2}) and \eqref{e:compevennew} show that (6) of Lemma \ref{l:obs}  
can be applied for $t=s_\beta \concat z$ and $u=s_\beta \concat \Phi(s_\beta)$ and $\delta=\delta_{s_\beta}+1$. This yields for every $z<\Phi(s_\beta)$ that
\begin{equation}
\label{e:compeven3}\max\{I(s_\beta \concat z)(\delta_{s_\beta}+1),S(s_\beta
\concat z)(\delta_{s_\beta}+1)\}\leq \end{equation}
\[
\leq I(s_\beta \concat
\Phi(s_\beta))(\delta_{s_{\beta}}+1)<I(s_{\beta}\concat \Phi(s_\beta))(\delta_{s_\beta})=\Phi(s_\beta),
\]
where the last inequality follows from the fact that $I(s_{\beta}\concat \Phi(s_\beta))$ is strictly decreasing and the equality comes from \eqref{e:compeven1}.

So by equations (\ref{e:compeven2}), (\ref{e:compeven3}) and (3) of Lemma \ref{l:obs} for an $x \in (I(s_{\beta}\concat \Phi(s_\beta))(\delta_{s_\beta}+1),\Phi(s_\beta))$ we obtain \[\Phi(s_\beta \concat x) \leq \max\{I(s_\beta
\concat x)(\delta_{s_\beta}+1),S(s_\beta \concat x)(\delta_{s_\beta}+1)\}\]\[\leq
I(s_\beta \concat \Phi(s_\beta))(\delta_{s_\beta}+1)<x.  \]
Thus, the second part of (\ref{e:tosolve}) holds for $x$. The first part is clear from $x<\Phi(s_\beta)$ and (\ref{e:sbind0}), hence $s_\alpha=s_\beta \concat x$ is an appropriate choice for (*).

{\noindent \textbf{Case 3.}} $\beta$ is odd and $\delta_{s_\beta}$ is even. \\Then $s_\beta$ has a least
element $\min s_\beta$, and by induction and (\ref{e:yassumpt}) $\min s_\beta>\Phi(s_\beta)=\Phi(s_\beta \concat \Phi(s_\beta))$. Now let
$x \in (\Phi(s_\beta),\min s_\beta)$. Then we have $s_\beta \concat x \in \sigma^*[0,1]$, so the first part of \eqref{e:tosolve} holds. Since $\beta$ is odd, we have $s_\beta \concat x <_{altlex} s_\beta \concat
\Phi(s_\beta)$. Therefore, from the fact that $\delta_{s_\beta}$ is even using part (b) of (5) of Lemma \ref{l:obs} it follows that
\begin{equation}
\label{e:compfinal}I(s_\beta \concat x)(\delta_{s_\beta}) \leq S(s_\beta \concat
x)(\delta_{s_\beta}) \leq S(s_\beta \concat \Phi(s_\beta))(\delta_{s_\beta})\end{equation}
\[
\leq \Phi(s_\beta \concat \Phi(s_\beta))=\Phi(s_\beta)<x,
\]
where the last $\leq$ uses (\ref{e:deltaequ}) while the equality comes from (\ref{e:yassumpt}).
Hence, using (1) of Lemma \ref{l:obs} we get $\delta_{s_\beta \concat x} \geq \delta_{s_\beta}$, so by (3) of Lemma \ref{l:obs} and \eqref{e:compfinal} we obtain
\[\Phi(s_\beta \concat x)\leq
\max\{I(s_\beta \concat x)(\delta_{s_\beta}),S(s_\beta \concat
x)(\delta_{s_\beta})\}<x,\] thus, again $x$ satisfies the second part of (\ref{e:tosolve}) so $s_\alpha =s_\beta \concat x$ is a good
choice for (*).

Thus, in any case we can carry out the induction.
 \end{proof}
 In order to prove the theorem just notice that Lemma \ref{l:compmain} gives an $\omega_1$-long $\subsetneqq$-increasing sequence of elements in $\sigma^*[0,1]$. But then
$\bigcup_{\alpha<\omega_1} s_\alpha$ would be an $\omega_1$-long decreasing
sequence of reals, which is a contradiction.
Therefore no completion of $\un$ can be embedded into itself and this finishes the
proof of the theorem.\end{proof}

\begin{remark}
Let $C$ be the following set:
\[\{\bar{x} \concat x_\xi \concat 0: \bar{x} \in \sigma^*[0,1], \xi \text{ is even, } l(\bar{x})=\xi+1,x_\xi \not =0\}.\]
The ordering $<_{altlex}$ extends to the set $C \cup \uni$ naturally and it is not hard to show that this ordering is complete. By Theorem \ref{t:complete} this is not representable in $\mc{B}_1$. However, one can show that this ordering does not contain $\omega_1$, $\omega^*_1$ and Suslin lines. Thus, we obtain another proof of \cite[Theorem 4.1]{marci2}.
\end{remark}

\section{Proof of Proposition \ref{p:bb}}
\label{s:last}

 \begin{customthm}{\ref{p:bb}}
  (\cite{KL})
Let $X$ be a Polish space and $f \in b\mathcal{B}^+_1(X)$. Then $\Phi(f)$ is defined, $\Phi(f) \in
\sigma^*bUSC^+$ and we have 
\begin{enumerate}[(1)]
	  \item  $f=\sum^*_{\beta<\alpha} (-1)^\beta
f_{\beta}+(-1)^{\alpha}g_\alpha$ for every $\alpha \leq \xi_{f}$,
	  \item $f_{\xi_f} \equiv0,$ 
      \item $f=\sum^*_{\alpha<\xi_f} (-1)^\alpha f_{\alpha}$.
         \end{enumerate}
 \end{customthm}

\begin{proof} 
First we show that $\Phi(f)$ is defined and $\Phi(f) \in \sigma^*bUSC^+$. In
order to prove this, we will show the following lemma.
\begin{lemma}
\label{l:fgmon}
The functions $g_\alpha$ and $f_\alpha$ (assigned to $f$ in Definition
\ref{d:konst}) are bounded nonnegative and the sequence $(f_{\alpha})$ is
decreasing.
\end{lemma}

\begin{proof}

It follows trivially from the definition of the upper regularization that if $g$
is an arbitrary function then
\begin{equation}
\label{e:ftriv1}
g \text{ is bounded} \Rightarrow \widehat{g} \text{ exists, is bounded and } \widehat{g} \geq_p g.
\end{equation}
Now we prove the statement of the lemma by induction on $\alpha$.
If $\alpha=0$ then $g_0=f$ and $f_0=\widehat{f}$, hence from $f \in
b\mc{B}_1^+(X)$ and (\ref{e:ftriv1}) clearly follows that $g_0$ and $f_0$ are
bounded nonnegative functions.

If $\alpha$ is a successor then by definition
$g_\alpha=\widehat{g_{\alpha-1}}-g_{\alpha-1}$ so by the second part of
(\ref{e:ftriv1}) we have $g_\alpha\geq_p 0$. Moreover, since $g_{\alpha-1}$ is
bounded  $\widehat{g_{\alpha-1}}$ is also bounded. Thus, $g_\alpha$ is the
difference of two bounded functions, therefore it is also bounded. Therefore, by
(\ref{e:ftriv1}) $f_\alpha$ exists (notice that we have defined the upper
regularization only for bounded functions) and also bounded and nonnegative.

Now we show that the sequence $(f_\alpha)$ is also decreasing. By the nonnegativity
of $g_{\alpha-1}$ we have $f_{\alpha-1}-g_{\alpha-1} \leq_p f_{\alpha-1}$, so
\[f_{\alpha}=\widehat{f_{\alpha-1}-g_{\alpha-1}} \leq_p
\widehat{f}_{\alpha-1}=f_{\alpha-1}.\]
For limit $\alpha$ we have 
\begin{equation}
\label{e:limitg}
g_\alpha=\inf\{g_\beta:\beta<\alpha \text{ and
$\beta$ is even}\},
\end{equation}
so clearly $g_\alpha \geq_p 0$ and $g_\alpha$ is bounded. Hence using again 
(\ref{e:ftriv1}) we obtain that $f_{\alpha}$ is bounded and nonnegative.

Now for every $\beta$ we have $g_{\beta} \leq_p f_{\beta}$. Therefore, if $\beta$
is an even ordinal and $\beta<\alpha$ then by \eqref{e:limitg} we have \[g_{\alpha}\leq_p g_\beta \leq_p
f_{\beta},\] so $f_{\alpha} =\widehat{g_\alpha} \leq_p
\widehat{f_{\beta}}=f_\beta$. But if $\beta$ is odd, then $\beta+1$ is even and
$\beta+1<\alpha$. Using \eqref{e:limitg} we obtain $g_\alpha \leq_p g_{\beta+1}$ hence by the definition of $f_\alpha$ and $f_{\beta+1}$ and the inductive hypothesis we have $f_\alpha \leq_p f_{\beta+1} \leq_p
f_\beta$. 
This finishes the proof of the lemma.
\end{proof}

Clearly, by the definition of upper regularization, the functions $f_{\alpha}$
are upper semicontinuous. Therefore, by Lemma \ref{l:fgmon} we obtain that
$(f_\alpha)$ is a decreasing sequence of nonnegative USC functions, so it must
stabilize for some countable ordinal $\xi_f$ (Lemma
\ref{l:embed}). Therefore, for every function in $f \in b\mathcal{B}^+_1(X)$ we
have that $\Phi(f)$ is defined and $\Phi(f) \in \sigma^*bUSC^+(X)$. 

Now we need the following lemma.
 \begin{lemma}
\label{l:sumb1}Let $(f_\alpha)_{\alpha<\xi} \in \sigma^* USC^+$. Then
$\sum^*_{\alpha<\xi} (-1)^\alpha f_{\alpha}$ is a Baire class 1 function.
\end{lemma}
\begin{proof} We prove the lemma by induction on $\xi$. 

First, if $\xi$ is a successor just use that Baire class $1$ functions are
closed under addition and subtraction.

Second, if $\xi$ is a limit, by definition of the alternating sums we have that
\[\scalebox{1.3}{$\Sigma$}_{\alpha < \xi}^* (-1)^\alpha f_{\alpha}=\sup \{
\scalebox{1.3}{$\Sigma$}_{\beta<\alpha}^* (-1)^{\beta} f_{\beta}: \alpha<\xi,
\alpha \text{ even}\}.\]
For even $\alpha<\xi$ we have \[\scalebox{1.3}{$\Sigma$}^*_{\beta<\alpha}
(-1)^\beta f_{\beta}=\scalebox{1.3}{$\Sigma$}^*_{\beta<\alpha+1} (-1)^\beta
f_{\beta}-f_\alpha. \tag{*}\]
Again, for even $\alpha$ \[\scalebox{1.3}{$\Sigma$}^*_{\beta<\alpha} (-1)^\beta
f_{\beta}+f_\alpha -f_{\alpha+1}=\scalebox{1.3}{$\Sigma$}^*_{\beta<\alpha+2}
(-1)^\beta f_{\beta}\] so since the sequence $(f_\alpha)_{\alpha<\xi}$ is
 decreasing the sequence $(\sum^*_{\beta<\alpha} (-1)^\beta
f_{\beta})_{\alpha \text{ even}}$ is  increasing. Similarly, the
sequence $(\sum^*_{\beta<\alpha+1} (-1)^\beta f_{\beta})_{\alpha \text{ even}}$
is  decreasing. Notice that if $(r_\beta)_{\beta<\alpha}$ and
$(t_\beta)_{\beta<\alpha}$ are  decreasing transfinite sequences of nonnegative
reals such that $r_\beta-t_\beta$ is increasing, then 
\[\sup\{r_\beta-t_\beta:\beta<\alpha\}=\inf\{r_\beta:\beta<\alpha\}-\inf\{
t_\beta:\beta<\alpha\}.\]
Therefore, applying $(*)$ and these facts we have
\[\sup\{\scalebox{1.3}{$\Sigma$}^*_{\beta<\alpha} (-1)^\beta
f_{\beta}:\alpha<\xi \text{
even}\}=\]\[\inf\{\scalebox{1.3}{$\Sigma$}^*_{\beta<\alpha+1} (-1)^\beta
f_{\beta}:\alpha<\xi \text{ even}\}-\inf\{f_\alpha:\alpha<\xi \text{ even}\}.\]
The infimum of USC functions is also USC, hence the right-hand side of the equation
is the difference of the infimum of a countable family of Baire class 1 functions and a USC function.
Therefore, $\sup\{\sum^*_{\beta<\alpha} (-1)^\beta f_{\beta}:\alpha<\xi \text{
even}\}$ is the infimum of a countable family of Baire class 1 functions. Moreover, by the inductive
hypothesis, this function is also the supremum of a countable family of Baire class 1 functions. Now, using the fact that a function is Baire class $1$ if and only if the preimage of every open set is $\mathbf{\Sigma}^0_2(X)$ it is easy to see that if a function $h$ is the infimum of a countable family of Baire class $1$ functions then for every $a \in \R$ we have that $h^{-1}((-\infty,a))$ is in $\mathbf{\Sigma}^0_2(X)$. Similarly, if $h$ is the supremum of a countable family of Baire class $1$ functions then the sets $h^{-1}((a,\infty))$ are also in $\mathbf{\Sigma}^0_2(X)$. But this implies that a function that is both an infimum and a supremum of countable families of Baire class $1$ functions is also Baire class $1$.  

So, as an
infimum and supremum of countable families of Baire class 1 functions, the function
$\sup\{\sum^*_{\beta<\alpha} (-1)^\beta f_{\beta}:\alpha<\xi \text{ even}\}$ is
also a Baire class 1 function, which completes the inductive proof.\end{proof}

Now we prove (1) of the Proposition by induction on $\alpha$.

For $\alpha=0$ this is clear. If $\alpha$ is a successor, then
$g_{\alpha-1}=f_{\alpha-1}-g_{\alpha}$, so \[f=\scalebox{1.3}{$\Sigma$}_{\beta < \alpha-1}^*  (-1)^\beta
f_{\beta}+(-1)^{\alpha-1}g_{\alpha-1}=\]\[ \scalebox{1.3}{$\Sigma$}_{\beta <
\alpha-1}^* (-1)^\beta f_{\beta}+(-1)^{\alpha-1}(f_{\alpha-1}-g_\alpha)=\scalebox{1.3}{$\Sigma$}_{\beta < \alpha}^* (-1)^\beta
f_{\beta}+(-1)^{\alpha}g_\alpha.\]

For limit $\alpha$ notice that we have by induction
for every even $\beta<\alpha$ 
\[f=\scalebox{1.3}{$\Sigma$}^*_{\gamma<\beta} (-1)^{\gamma}
f_{\gamma}+g_{\beta}.\]
Then, using that the sequence $(f_\beta)_{\beta<\alpha}$ is
decreasing, the sequence $(\scalebox{1.3}{$\Sigma$}^*_{\gamma<\beta}
(-1)^{\gamma} f_{\gamma})_{\beta \text{ even}}$ is  increasing, so
$(g_\beta)_{\beta \text{ even}}$ is  decreasing as
their sum is constant $f$. 

Notice that if  $(r_\beta)_{\beta<\alpha}$ is an increasing and
$(t_\beta)_{\beta<\alpha}$ is a decreasing transfinite sequence of nonnegative
reals such that $r_\beta+t_\beta=c$ is constant, then 
\[c=\sup\{r_\beta+t_\beta:\beta<\alpha\}=\sup\{r_\beta:\beta<\alpha\}+\inf\{
t_\beta:\beta<\alpha\}.\]
So 
\[f=\sup_{\beta \text{ even}, \beta<\alpha}
 \left(\scalebox{1.3}{$\Sigma$}^*_{\gamma<\beta} (-1)^{\gamma}
f_{\gamma}+g_\beta \right)=\]
\[\sup_{\beta \text{ even},
\beta<\alpha}\scalebox{1.3}{$\Sigma$}^*_{\gamma<\beta} (-1)^{\gamma}
f_{\gamma}+\inf_{\beta \text{ even}, \beta<\alpha}g_{\beta}=
\scalebox{1.3}{$\Sigma$}^*_{\beta<\alpha} (-1)^\beta f_{\beta}+g_\alpha,\]
where the last equality follows from the definition of   $\sum^*_{\beta<\alpha}
(-1)^\beta f_{\beta}$ and $g_\alpha$.

This proves the induction hypothesis, so we have $(1)$. 

After rearranging the equality in (1) we have that
\[(-1)^{\alpha+1}g_\alpha=\scalebox{1.3}{$\Sigma$}^*_{\beta<\alpha} (-1)^\beta
f_{\beta}-f.\]
By Lemma \ref{l:sumb1} we have that the sum on the right-hand side of the equation is
a Baire class $1$ function, therefore $g_{\alpha}$ is also Baire class $1$. We have that $f_{\xi_f+1}\equiv f_{\xi_f}$, so by Definition \ref{d:konst} we have $\widehat{\widehat{g_{\xi_f}}-g_{\xi_f}}=\widehat{g_{\xi_f}}$.
Hence in order to prove (2) it is enough to show the following claim.\\
\textbf{Claim.} \textit{If $g$ is a nonnegative, bounded Baire class $1$
function such that $\widehat{g}=\widehat{\widehat{g}-g}$ then $g \equiv 0$. }

\begin{proof}[Proof of the Claim.] 
Suppose the contrary. Then there exists an $\varepsilon>0$ such that
$\{x:g(x)>\varepsilon\} \not = \emptyset$. Let
$K=\overline{\{x:g(x)>\varepsilon\}}$. Since $g$ is a Baire class $1$ function
we have that there exists an open set 
$V$ such that \[\varepsilon>osc(g, K\cap V)\ \ (=\sup_{x,y \in K \cap V}
|g(x)-g(y)|) \] and $K \cap V$ is not empty (see \cite[24.15]{kech}).

The function $\limsup_{y \to x} g(y)$ (here in the $\limsup$ we do not exclude
those sequences which contain $x$) is USC. Therefore, by definition 
$\widehat{g} \leq_p \limsup g$. Hence letting $h=\widehat{g}-g$ we have that 
\begin{equation}
\label{e:hest}
h \leq_p \limsup (g) - g.
\end{equation}
Now, we claim that \begin{equation}
\label{e:hest2}
(\limsup (g)-g)|_{V \cap K} \leq \varepsilon.
\end{equation}
Suppose the contrary. Then there exists an $x \in V \cap K$ such that $(\limsup_{y
\to x} g(x))-g(x) > \varepsilon$. Consequently, there exists a sequence $y_n \to
x$, such that $\lim_{n \to \infty} g(y_n)>g(x)+\varepsilon$. Using the nonnegativity
of $g$ and the fact that $g|_{K^c} \leq \varepsilon$ we get that $y_n \in K \cap
V$ except for finitely many $n$'s.  But then $osc(g,K\cap V)>\varepsilon$, a
contradiction. So we have (\ref{e:hest2}) and using (\ref{e:hest}) we obtain 
\begin{equation}
\label{e:hest3}
h|_{V \cap K} \leq \varepsilon.
\end{equation}
Observe now that if for a bounded function $f$ and an open set $U$ we have that
$f|_U \leq \varepsilon$, then $\widehat{f}|_U \leq  \varepsilon$ (clearly, if
$|f|<K$ then the function $K \cdot\chi_{U^c}+\varepsilon \cdot\chi_{U}$ is an
USC upper bound of $f$).

By the above observation used for $g$ on $K^c$ we have that 
$\widehat{g}|_{K^c}\leq \varepsilon$, in particular from
$h=\widehat{g}-g\leq_p \widehat{g}$ we obtain that $h|_{K^c}\leq \varepsilon$.
Then from (\ref{e:hest3}) we get $h|_{V} \leq \varepsilon$. So finally, using
the above observation for $h$ and $V$ we obtain $\widehat{h}|_V \leq \varepsilon$. 

The set $\{x:g(x)>\varepsilon\}$ is dense in $K$, hence there exists an $x_0 \in
V \cap \{x:g(x)>\varepsilon\}$. On the one hand $\widehat{g}(x_0) \geq g(x_0) >
\varepsilon$, on the other by $x \in V$ we get $\widehat{h}(x_0)\leq
\varepsilon$. This contradicts the assumption that $\widehat{g}=\widehat{h}$.\end{proof}
So we have proved $(2)$ of Proposition \ref{p:bb}. 

$(3)$ easily follows from Lemma \ref{l:fgmon}, $(1)$, $(2)$ since $0 \leq g_{\xi_f}
\leq f_{\xi_f} \equiv 0$. This finishes the proof of the proposition.\end{proof}

\section{Open problems}
\label{s:open}

Probably the most natural and intriguing problem is the following. Recall that the $\alpha$th level of the Baire hierarchy in a space $X$ is denoted by $\mc{B}_\alpha(X)$. Unless stated otherwise, $X$ is an uncountable Polish space.

\begin{problem}
\label{q:1}
Let $2 \leq \alpha <\omega_1$. Characterize the order types of the linearly ordered subsets of $\mc{B}_\alpha(X)$. For instance, does there exist a (simple) universal linearly ordered set for $\mc{B}_\alpha(X)$? And how about the class of Borel measurable functions $\cup_{\alpha<\omega_1} \mc{B}_\alpha(X)$?
\end{problem}

We remark here that Komj\'ath \cite{komjath} proved that under the Continuum Hypothesis every linearly ordered set of cardinality at most $\mathfrak{c}$ can be represented in $\mc{B}_2(X)$ (hence in $\mc{B}_\alpha(X)$ for any $\alpha \geq 2$ as well). Nevertheless, a $ZFC$ result would be very interesting and in light of our solution to Laczkovich's problem now it seems conceivable that one can construct relatively simple universal linearly ordered sets in these cases as well. As a first step in this direction it would be interesting to see if the result of Kechris and Louveau can be generalized to $\mc{B}_\alpha(X)$. Actually, closely related results from this paper have already been generalised from the Baire class $1$ case to the Baire class $\alpha$ case in \cite{ekv}.

Let $(L_n)_{n \in \omega}$ and $L$ be linearly ordered sets. We say that 
 $L$ is a \emph{blend of $(L_n)_{n \in \omega}$} if $L$ can be partitioned to pairwise disjoint subsets $(L'_n)_{n \in \omega}$ such that $L_n$ is order isomorphic to $L'_n$ for every $n$. Elekes \cite{marci1} proved that if the duplication and completion of every representable ordering was representable then countable blends of representable orderings would also be representable. As we have seen (Theorem \ref{t:complete}), the second condition of this theorem fails, hence it is quite natural to ask the following.

\begin{problem}
Suppose that the linearly ordered sets $L_n$ are representable in $\mc{B}_1(X)$ and $L$ is a blend of $(L_n)_{n \in \omega}$. Does it follow that $L$ is also representable in $\mc{B}_1(X)$?
\end{problem}

The authors would expect a negative answer using similar ideas and techniques as in the proof of Theorem \ref{t:complete}.

Elekes and Kunen \cite{marci3} investigated Problem \ref{p:laczk} in general, for non-Polish $X$. This raises the next question:

\begin{problem}
Let $X$ be a topological space (e. g. a separable metric space). Characterize the order types of the linearly ordered subsets of $\mc{B}_1(X)$. For instance, does there exist a (simple) universal linearly ordered set for $\mc{B}_1(X)$?
\end{problem}

We believe that an affirmative answer might be useful in answering Question \ref{q:1} using topology refinements.

The next problem concerns characterizing all the subposets of our function spaces instead of only the linearly ordered ones. For example, it is not hard to check that $\mathcal{F}(X) = \mathcal{C}([0,1])$ contains an isomorphic copy of a poset $P$ iff $(\mathcal{P}(\omega), \subsetneqq)$ does. 

\begin{problem}
Characterize, up to poset-isomorphism, the subsets of $\mc{B}_1(X)$. Does there exist a simple, informative universal poset? For instance, is $\mathbf{\Delta}^0_2(X)$ or $USC^{<\omega_1}_{\searrow 0}(X)$ universal?
\end{problem}

Here $USC^{<\omega_1}_{\searrow 0}$ is defined analogously to $\uni$ and is ordered by the natural modification of $<_{altlex}$.
Notice that our method of proving that $(\mc{B}_1(X),<_p) \hookrightarrow (\mathbf{\Delta}^0_2(X),\subsetneqq)$ does not give a poset isomorphism between $\mc{B}_1(X)$ and its image. In fact, the image is linearly ordered. Unfortunately, it can be easily seen that even the Kechris-Louveau-type embedding $\mc{B}_1(X) \to bUSC^{<\omega_1}_{\searrow 0}$, that is, assigning to every Baire class $1$ function its canonical resolution as a sum is not a poset isomorphism.



At first sight Laczkovich's problem seems to be closely related to the theory of Rosenthal
compacta \cite{ros}, \cite{enc}, \cite{tod2}.

\begin{problem}
Explore the connection between the topic of our paper and the theory of Rosenthal compacta.
\end{problem}

\subsection*{Acknowledgements} 
	We are very grateful to Mikl\'os Laczkovich and Stevo Todor\v{c}evi\' c for numerous illuminating conversations. We would also like to thank the referee for helpful remarks.

\end{document}